\documentclass[11pt,reqno]{amsart}
\usepackage{amsmath,amsthm,amssymb}

\newtheorem{theorem}{Theorem}[section]
\newtheorem{lemma}[theorem]{Lemma}
\newtheorem{proposition}[theorem]{Proposition}
\newtheorem{corollary}[theorem]{Corollary}

\newtheorem{problem}[theorem]{Problem}

\theoremstyle{definition}

\newtheorem{definition}[theorem]{Definition}

\newtheorem{remark}[theorem]{Remark}

\newtheorem{example}[theorem]{Example}

\numberwithin{equation}{section}

\begin{document}

\title{Twists, Codazzi Tensors, and the $6$-sphere}

\author{David N. Pham}
\address{Department of Mathematics $\&$ Computer Science, QCC CUNY, Bayside, NY 11364}
\curraddr{}
\email{dnpham@qcc.cuny.edu}
\thanks{This work was supported by PSC-CUNY Award $\#$68045-00 56.}


\dedicatory{}

\begin{abstract}
Let $(M,g,J,\omega)$ be an almost Hermitian manifold.  Given an automorphism $\psi\in \mathrm{Aut}(TM)$, the existing structure can be twisted to obtain a new almost Hermitian manifold $(M,g^\psi,J^\psi,\omega^\psi)$.  In the current paper, we study these $\psi$-twisted almost Hermitian structures with particular emphasis on questions regarding the integrability of $J^\psi$ and the Riemannian geometry of $g^\psi$.  By studying the latter, we identity a certain class of $\mathrm{Aut}(TM)$ with nice transformation properties.  We call these automorphisms $g$-\textit{Codazzi maps} because of their close relationship with Codazzi tensors.  The aforementioned results are ultimately applied to the standard nearly K\"{a}hler structure on the $6$-sphere where we prove a nonintegrability result for the class of $g$-Codazzi maps.
\end{abstract}


\maketitle

\section{Introduction}
The study of deformations of almost Hermitian structures is a very active area of research in complex geometry, especially as it pertains to questions of stability of certain special geometries (e.g., nearly K\"{a}hler, Calabi-Yau, SKT, etc.)  \cite{MNS2008, Foscolo2017,Bogomolov1978,Tian1987,Todorov1989,FT2009,ST2010}.   In the current paper, we consider ``twists" of an almost Hermitian manifold $(M,g,J,\omega)$ by elements of $\mathrm{Aut}(TM)$ as opposed to genuine deformations.  Our focus is then on questions concerning the integrability of the twisted almost complex structure as well as the Riemannian geometry of the twisted metric.  While these twists are far simpler and more limited than deformations, we believe that they still carry with them interesting phenomena and questions worthy of study.  We will motivate this view with an example.  First, however, we make a formal definition:
\begin{definition}
\label{defPsiTwist}
Let $(M,g,J,\omega)$ be an almost Hermitian manifold and let $\psi\in \mbox{Aut}(TM)$.  The $\psi$-twist of $(M,g,J,\omega)$ is the almost Hermitian manifold
$$
(M,g^\psi,J^\psi,\omega^\psi)
$$ 
where
$$
g^\psi:=g(\psi^{-1}\cdot,\psi^{-1}\cdot),\hspace*{0.1in} J^\psi:=\psi\circ J\circ \psi^{-1},\hspace*{0.1in}\omega^\psi:=\omega(\psi^{-1}\cdot,\psi^{-1}\cdot).
$$
\end{definition}
\noindent From Definition \ref{defPsiTwist}, it is very easy to see that the twisted structure is still an almost Hermitian manifold.  

For a motivating example, consider the classic left-invariant strictly nearly K\"{a}hler structure on $S^3\times S^3=SU(2)\times SU(2)$ \cite{BVW2015, Gray1972, Butruille2006}.  Since the structure is left-invariant, everything is completely defined at the Lie algebra level.  The Lie algebra $\mbox{Lie}(SU(2)\times SU(2))=\mathfrak{su}(2)\oplus \mathfrak{su}(2)$ has basis 
$$
e_1,e_2,e_3,f_1,f_2,f_3
$$
with bracket relations
\begin{align*}
&[e_1,e_2]=e_3,~[e_2,e_3]=e_1,~[e_3,e_1]=e_2,\\
&[f_1,f_2]=f_3,~[f_2,f_3]=f_1,~[f_3,f_1]=f_2,\\
&[e_i,f_j]=0.
\end{align*}
For convenience, let $\mathfrak{g}:=\mathfrak{su}(2)\oplus \mathfrak{su}(2)$.  The almost complex structure $J:\mathfrak{g}\rightarrow \mathfrak{g}$ is given as follows:
$$
Je_i=-\frac{1}{\sqrt{3}}(e_i+2f_i),\hspace*{0.2in} Jf_i=\frac{1}{\sqrt{3}}(2e_i+f_i).
$$
The metric $g$ is given by
$$
g(e_i,e_j)=g(f_i,f_j)=\frac{4}{3}\delta_{ij},\hspace*{0.1in}g(e_i,f_j)=-\frac{2}{3}\delta_{ij}.
$$
The left-invariance of $g$ implies that its Levi-Civita connection $\nabla^g$ is also left-invariant.  Write $\nabla^g_{e_i}e_j=\sum_k\Gamma^k_{ij}e_k$.  Using the Kozul formula one finds that the nonzero components of $\nabla^g$ are given by
$$
\Gamma_{12}^3=\frac{1}{2},~\Gamma_{13}^2=-\frac{1}{2},~\Gamma_{15}^3=-\frac{1}{6},~\Gamma_{15}^6=\frac{1}{6},~\Gamma_{16}^2=\frac{1}{6},~\Gamma_{16}^5=-\frac{1}{6},
$$
$$
\Gamma_{21}^3=-\frac{1}{2},~\Gamma_{23}^1=\frac{1}{2},~\Gamma_{24}^3=\frac{1}{6},~\Gamma_{24}^6=-\frac{1}{6},~\Gamma_{26}^1=-\frac{1}{6},~\Gamma_{26}^4=\frac{1}{6},
$$
$$
\Gamma_{31}^2=\frac{1}{2},~\Gamma_{32}^1=-\frac{1}{2},~\Gamma_{34}^2=-\frac{1}{6},~\Gamma_{34}^5=\frac{1}{6},~\Gamma_{35}^1=\frac{1}{6},~\Gamma_{35}^4=-\frac{1}{6},
$$
$$
\Gamma_{42}^3=\frac{1}{6},~\Gamma_{42}^6=-\frac{1}{6},~\Gamma_{43}^2=-\frac{1}{6},~\Gamma_{43}^5=\frac{1}{6},~\Gamma_{45}^6=\frac{1}{2},~\Gamma_{46}^5=-\frac{1}{2},
$$
$$
\Gamma_{51}^3=-\frac{1}{6},~\Gamma_{51}^6=\frac{1}{6},~\Gamma_{53}^1=\frac{1}{6},~\Gamma_{53}^4=-\frac{1}{6},~\Gamma_{54}^6=-\frac{1}{2},~\Gamma_{56}^4=\frac{1}{2},
$$
$$
\Gamma_{61}^2=\frac{1}{6},~\Gamma_{61}^5=-\frac{1}{6},~\Gamma_{62}^1=-\frac{1}{6},~\Gamma_{62}^4=\frac{1}{6},~\Gamma_{64}^5=\frac{1}{2},~\Gamma_{65}^4=-\frac{1}{2}.
$$
The diligent reader can verify that $(g,J)$ is an almost Hermitian structure which satisfies the nearly K\"{a}hler condition
$$
(\nabla^g_{e_i}J)e_j=-(\nabla^g_{e_j}J)e_i,\hspace*{0.1in}\forall~i,j.
$$
Now define a (left-invariant) bundle automorphism $\psi$ of $T(SU(2)\times SU(2))$ by
\begin{align*}
\psi e_1=e_1,\hspace*{0.1in}~\psi e_2=f_1,&\hspace*{0.1in}~\psi e_3 = e_3,\\
\psi f_1=-\frac{1}{2}e_1-\frac{\sqrt{3}}{2}e_2,\hspace*{0.1in}\psi f_2=-\frac{1}{2} &f_1-\frac{\sqrt{3}}{2}f_2,\hspace*{0.1in}\psi f_3=-\frac{1}{2}e_3-\frac{\sqrt{3}}{2}f_3.
\end{align*}
By direct calculation, one verifies that the Nijenhuis tensor $N_{J^\psi}$ of $J^\psi$ is now precisely zero, where we recall that 
$$
N_{J^\psi}(X,Y):=J^\psi[J^\psi X,Y]+J^\psi [X,J^\psi Y]+[X,Y]-[J^\psi X,J^\psi Y].
$$
By the Newlander-Nirenberg Theorem, $J^\psi$ is an integrable almost complex structure.  In addition, the twist has also transformed the metric in a surprising way.  One also verifies that the fundamental form $\omega^\psi$ is pluriclosed, that is, $\partial\overline{\partial}\omega^\psi=0$, where  $\partial, \bar{\partial}$ are the Dolbeault operators with respect to $J^\psi$. Note that the pluriclosed condition is equivalent to the condition $dc=0$, where $c$ is the 3-form defined by
$$
c:=(d\omega^\psi)(J^\psi\cdot,J^\psi\cdot,J^\psi\cdot).
$$ 
The condition $dc=0$ is easier to verify in the current example.  Hence, the $\psi$-twisted Hermitian structure $(g^\psi,J^\psi,\omega^\psi)$ is an SKT manifold.  The reader with little or no familiarity with SKT geometry is referred to \cite{FinoGrantcharov2024} and the references therein for a nice overview of the subject.

In the above motivating example, a non-integrable almost complex structure has been $\psi$-twisted into an integrable one (in addition to producing an SKT metric in the process).  Before continuing, let us take a moment to consider the differences between an automorphism $\psi\in \mathrm{Aut}(TM)$ and a diffeomorphism $F\in \mathrm{Diff}(M)$.  In both cases, one has induced maps
$$
\psi: \mathfrak{X}(M)\rightarrow \mathfrak{X}(M),\hspace*{0.1in} dF: \mathfrak{X}(M)\rightarrow \mathfrak{X}(M).
$$
However, it would be a mistake to regard $\psi$ and $dF$ as similar objects.  Indeed, the induced map on vector fields given by $\psi$ is $C^\infty(M)$-linear while the induced map given by $dF$ is not $C^\infty(M)$-linear.  This is the reason why $\psi$ is incapable of preserving the Lie bracket of vector fields (unlike $dF$).  Indeed, let $X,Y\in \mathfrak{X}(M)$ and let $f\in C^\infty(M)$ and suppose that $\psi$ preserves Lie brackets.  Then 
$$
[\psi(X),\psi(fY)]=(\psi(X)f)\psi(Y)+f[\psi(X),\psi(Y)]
$$
and 
$$
\psi([X,fY])=(Xf)\psi(Y)+f\psi([X,Y]).
$$
Under the assumption that $\psi$ preserves Lie brackets, it follows that 
$$
\psi(X)f=Xf
$$
for all $f\in C^\infty(M)$ and $X\in \mathfrak{X}(M)$. From this, we see that $\psi$ fails to preserve the Lie bracket of vector fields unless $\psi = \mathrm{id}_{TM}$.  

Using the differential $dF$, one can certainly conjugate an existing almost complex structure $J$ on $M$ to obtain a new almost complex structure 
$$
I:=dF\circ J\circ dF^{-1}.
$$ 
Note that while $dF$ and $dF^{-1}$ are not $C^\infty(M)$-linear, the above expression is $C^\infty(M)$-linear and defines a genuine almost complex structure on $M$.  By direct calculation, the Nijenhuis tensor of $I$ transforms as
$$
dF^{-1}\left(N_I(X,Y)\right)=N_J(dF^{-1}(X), dF^{-1}(Y)).
$$
Hence, $N_I=0$ if and only if $N_J=0$.  In the nearly K\"{a}hler example given above, the fact that $\psi$ \textit{fails} to preserve Lie brackets is precisely the reason why the \textit{possiblity} exists to transform a non-integrable almost complex structure into an integrable one.

With this motivation, the current paper will study twists of an almost Hermitian manifold $M$ by elements of $\mathrm{Aut}(TM)$, especially as it pertains to questions related to the integrability of the almost complex structure.  Taking inspiration from the above example on $S^3\times S^3$ and the work of Blanchard \cite{Blanchard1953}, LeBrun \cite{LeBrun1987}, Bor$\backslash$Hernandez-Lamoneda \cite{BorHernandezLamoneda1999}, and Kruglikov \cite{Kruglikov2017} on the 6-sphere, we consider the following problem: 
\begin{problem}
\label{conjPsiTwist}
Let $J$ be the standard almost complex structure on $S^6$.  
\begin{itemize}
\item[(i)] Determine all $\psi\in \mathrm{Aut}(TS^6)$ for which $J^\psi$ is nonintegrable.
\item[(ii)] Determine all $\psi \in \mathrm{Aut}(TS^6)$ for which $J^\psi$ is integrable.
\end{itemize}
\end{problem}
\noindent The following characterization\footnote{The author wishes to thank Xavier Butler for kindly reminding the author of this important fact.} is well known to complex geometers.  However, the author is unable to find an explicit reference for it.
\begin{proposition}
\label{propACSonS6}
Let $J$ be the standard almost complex structure on $S^6$.  Then every almost complex structure on $S^6$ is of the form $\pm J^\psi$ for $\psi\in \mbox{Aut}(TS^6)$.
\end{proposition}
\begin{proof}
Let $I$ be any almost complex structure on $S^6$ which induces the same canonical orientation as $J$.  Recall that the canonical orientation given I at $p\in S^6$ is given by choosing any $u_1,u_2,u_3\in T_pS^6$ for which 
$$
\{u_1,Iu_1,u_2,Iu_2,u_3,Iu_3\}
$$
is a basis of $T_pS^6$.  The canonical orientation given by $I$ at $p$ is then the orientation defined by the above basis.  Let $BU(3)$ be the classifying space associated to $U(3)$.  Recall that the isomorphism classes of complex vector bundles over $S^6$ ($\mathrm{Vect}_{\mathbb{C}}(S^6)$)  is in one to one correspondence with the space of homotopy classes of maps \cite{MilnorStasheff1974}
$$
[S^6,BU(3)].
$$
Note that each homotopy class in $[S^6,BU(3)]$ can always be represented by a smooth map $f:S^6\longrightarrow BU(3)$.  Then the aforementioned correspondence is given by
$$
[S^6,BU(3)]\stackrel{\sim}{\longrightarrow}\mathrm{Vect}_{\mathbb{C}}(S^6),\qquad f\mapsto f^\ast \gamma_3
$$
where $\gamma_3\longrightarrow BU(3)$ is the universal complex vector bundle of rank $3$.  

By Bott peirodicity, one has \cite{Milnor1963}
$$
[S^6,BU(3)]=\pi_6(BU(3))\simeq \pi_5(U(3))\simeq \mathbb{Z}.
$$
Let
$$
c_j\in H^{2j}(BU(3),\mathbb{Z})
$$
be the $j$-th Chern class associated to $\gamma_3$.  Then the $j$-th Chern class associated to a complex vector bundle $E\simeq f^\ast \gamma_3$ of rank $3$ with classifying map $f:S^6\longrightarrow BU(3)$ is
$$
c_j(E):=f^\ast c_j\in H^{2j}(S^6,\mathbb{Z}).
$$
Consider the map
$$
\varphi: \pi_6(BU(3))\longrightarrow \mathbb{Z},\qquad f\mapsto \int_{S^6}c_3(f^\ast\gamma_3),
$$
where we fix the orientation on $S^6$ to be the standard one (which coincides with the canonical orientation of $J$).

Let $(TS^6,J)$ and $(TS^6,I)$ denote the complex vector bundle given by $iX:=JX$ and $iX:=IX$ respectively for $X\in TS^6$.  Then one has 
$$
c_3(TS^6,J)=c_3(TS^6,I)=e(TS^6)
$$
where $e(TS^6)$ is the Euler class of $TS^6$ (cf. \cite{MilnorStasheff1974}) with $TS^6$ oriented via the the canonical orientation of $J$ (which is the same as that of $I$ by hypothesis).  Let $f: S^6\longrightarrow BU(3)$ be the classifying map associated to $(TS^6,J)$.  Then 
$$
\varphi(f)=\int_{S^6}c_3(TS^6,J)=\int_{S^6}e(TS^6)=\chi(S^6)=2.
$$
Let $f_0$ be a generator of $\pi_6(BU(3))\simeq \mathbb{Z}$.  Then $f=kf_0$ for some $k\in \mathbb{Z}$.  Since $\varphi$ is well known to be a group homomorphism (cf. \cite{MilnorStasheff1974}), we have
$$
\varphi(f)=\varphi(kf_0)=k\varphi(f_0)=2
$$
which implies that $\varphi(f_0)\neq 0$.  Hence, $\varphi$ is injective and we see that each distinct element of
$$
\mathrm{Vect}_{\mathbb{C}}(S^6)\stackrel{1-1}{\leftrightarrow} \pi_6(BU(3))
$$
corresponds to a unique integer under $\varphi$.  However, if $f': S^6\longrightarrow BU(3)$ is a classifying map for $(TS^6,I)$, we also have
$$
\varphi(f')=2.
$$
Hence, $(TS^6,J)\simeq (TS^6,I)$ as complex vector bundles which means there exists an automorphism $\psi\in \mbox{Aut}(TS^6)$ such that $I\psi=\psi J$ or, equivalently, $I=J^\psi$.   

From the definition of the canonical orientation of an almost complex structure, we see that $-J$ induces the orientation opposite that of $J$ and 
$$
c_3(TS^6,-J)=-c_3(TS^6,J).
$$
Hence, if $h: S^6\longrightarrow BU(3)$ is a classifying map for $(TS^6,-J)$, we have $\varphi(h)=-2$ which shows that $(TS^6,-J)$ and $(TS^6,J)$ are not isomorphic as complex vector bundles.  So if $K$ is another almost complex structure on $S^6$ with orientation same as that of $-J$, the same argument given above gives $(TS^6,-J)\simeq (TS^6,K)$.  Combining these two cases, it follows that every possible almost complex structure on $S^6$ is of the form $\pm J^\psi$ for $\psi\in \mbox{Aut}(TS^6)$.

\end{proof}
\noindent  The recent Claude AI-assisted work of Levent Alp\"{o}ge \cite{AC26} is currently being studied by the mathematical community.  However, there is strong optimism that the proof is correct which would resolve the Hopf problem in the positive.  Alp\"{o}ge's proof does not explicitly construct the integrable almost complex structure on $S^6$.   However, if the the proof turns out to be correct, then $S^6$ does admit an integrable almost complex structure.  Proposition \ref{propACSonS6} then implies that there exists a $\psi\in \mbox{Aut}(TS^6)$ for which $J^\psi$ is integrable.  It would be quite astounding if such a $\psi$ could be given explicitly.

While we are unable to resolve Problem \ref{conjPsiTwist} at the present time, we are able to find a definitive solution for a special class of automorphisms which we call $g$-\textit{Codazzi maps} (see Definition \ref{defgCodazzi} and Theorem \ref{thmS6Twist}).  The notion of a $g$-Codazzi map arises naturally when one considers the Levi-Civita connection of an arbitrary $\psi$-twisted Riemannian metric and observes that the aforementioned formula for the Levi-Civita connection simplifies drastically when one imposes the Codazzi condition (see Proposition \ref{propLeviCivita}).   

The rest of the paper is organized as follows.   In Section \ref{SectionTwistedMetrics}, we study the Riemannian geometry of $\psi$-twisted metrics with particular emphasis on $g$-Codazzi maps.  In Section \ref{secgCodazziSn}, we study $g$-Codazzi maps on the round sphere and prove a number of results.  Lastly, in Section \ref{SectionIntSixSphere}, we study the integrability of almost complex structures twisted by $g$-Codazzi maps.  The results of this section are combined with those of Section \ref{secgCodazziSn} to obtain a complete resolution to Problem \ref{conjPsiTwist} for the special case of $g$-Codazzi maps.  We note that the beautiful result of  Bor$\backslash$Hernandez-Lamoneda in \cite{BorHernandezLamoneda1999}, when used in conjunction with our curvature operator formula (see Proposition \ref{propCurvatureOperator}), is only able to provide a partial solution to Problem \ref{conjPsiTwist} for the special case of $g$-Codazzi maps.  The reason for this limitation is due to the fact that the result of \cite{BorHernandezLamoneda1999} is formulated in terms of an inequality condition involving the eigenvalues of the curvature operator.  When this condition is satisfied, one can conclude that the twisted almost complex structure is nonintegrable.  However, one can construct $g$-Codazzi maps which fail to satisfy the aforementioned inequality condition. When this occurs, the result of \cite{BorHernandezLamoneda1999} is no longer applicable.  We demonstrate this later with an example (along with other concrete examples to illustrate our results).

\section{The Riemannian geometry of twisted metrics}
\label{SectionTwistedMetrics}
\noindent We begin by fixing our notation and conventions.  Let $(M,g,J,\omega)$ be an almost Hermitian manifold with metric $g$, almost complex structure $J$, and fundamental form $\omega:=g(J\cdot,\cdot)$.  We let
$$
\flat_g: TM\stackrel{\sim}{\rightarrow} T^\ast M,\hspace*{0.1in} X\mapsto X^{\flat_g}
$$
$$
\sharp_g:T^\ast M\stackrel{\sim}{\rightarrow} TM,\hspace*{0.1in} \alpha \mapsto \alpha^{\sharp_g}
$$
denote the musical isomorphisms induced by $g$.  $\flat_g$ and $\sharp_g$ are extended naturally to $\wedge^kTM$ and $\wedge^k T^\ast M$ respectively by letting them act on each factor of the wedge product.  The metric $g$ is extended to $\wedge^k TM$ in the usual way and we let $g^{-1}$ denote the induced metric on $\wedge^k T^\ast M$.  

We denote the Levi-Civita connection of $g$ by $\nabla^g$.  We define the curvature endomorphism $R^g\in \Gamma(\wedge^2T^\ast M\otimes\mathrm{End}(TM))$ by
$$
R^g(X,Y)Z:=\nabla^g_X\nabla^g_YZ-\nabla^g_Y\nabla^g_XZ-\nabla^g_{[X,Y]}Z.
$$
The Riemann curvature tensor is denoted as $\mathrm{Rm}^g$ and is defined by
 $$
 \mathrm{Rm}^g(X,Y,Z,W):=g(R^g(X,Y)Z,W).
 $$
 We recall that $\mathrm{Rm}^g$ has the following symmetries:
 \begin{itemize}
 \item[(i)] $\mathrm{Rm}^g(Y,X,Z,W)=-\mathrm{Rm}^g(X,Y,Z,W)$
 \item[(ii)] $\mathrm{Rm}^g(X,Y,W,Z)=-\mathrm{Rm}^g(X,Y,Z,W)$
 \item[(iii)] $\mathrm{Rm}^g(X,Y,Z,W)=\mathrm{Rm}^g(Z,W,X,Y)$.
 \end{itemize}
 As a consequence of this, we can also regard $\mathrm{Rm}^g$ as a symmetric bilinear form on $\wedge^2 TM$.  We define the curvature operator on 2-forms 
 $$
 \mathcal{R}^g: \wedge^2 T^\ast M\rightarrow \wedge^2 T^\ast M
 $$
 by
 $$
 g^{-1}(\beta,\mathcal{R}^g(\gamma))=-\mbox{Rm}^g(\beta^{\sharp_g},\gamma^{\sharp_g})
 $$
 for $\beta,\gamma\in \Omega^2(M)$.  Since $\mathrm{Rm}^g$ is also a symmetric bilinear form on $\wedge^2 TM$, we note that $\mathcal{R}^g$ is also self-adjoint with respect to $g^{-1}$.  Equivalently, we may also define $\mathcal{R}^g$ by
$$
(\mathcal{R}^g(\gamma))(X,Y):=-\mathrm{Rm}^g(X\wedge Y,\gamma^{\sharp_g})
$$
for $X,Y\in \mathfrak{X}(M)$.  Lastly, if $\psi\in \mathrm{End}(TM)$, we denote its adjoint with respect to $g$ by $\psi^\dagger$, that is,
$$
g(\psi X,Y)=g(X,\psi^\dagger Y).
$$

With our notation and conventions finally fixed, we now begin with the following result which relates the Levi-Civita connection of $g^\psi$ with that of $g$ where $\psi\in \mathrm{Aut}(TM)$.
\begin{proposition}
\label{propLeviCivita}
Let $\psi\in \mbox{Aut}(TM)$ and let $h:=g^\psi$.  Then
\begin{align}
\nonumber
&2h(\nabla^h_XY,Z)=2h(\nabla^g_XY,Z)+h(\psi(\nabla^g_X\psi^{-1})Y,Z)\\
\nonumber
&+h(\psi(\nabla^g_Y\psi^{-1})X,Z)+h(\psi(\nabla^g_X\psi^{-1})Z,Y)\\
\nonumber
&-h(\psi(\nabla^g_Z\psi^{-1})X,Y)+h(\psi(\nabla^g_Y\psi^{-1})Z,X)\\
\label{eqLCA}
&-h(\psi(\nabla^g_Z\psi^{-1})Y,X).
\end{align}
In particular, if $(\nabla^g_X\psi^{-1})Y=(\nabla^g_Y\psi^{-1})X$, then
\begin{align}
\label{eqLCB}
\nabla^h_XY=\psi\nabla^g_X(\psi^{-1}Y).
\end{align}
\end{proposition}
\begin{proof}
Using the Kozul formula, we have
\begin{align}
\nonumber
2h(\nabla^h_XY,Z)&=Xh(Y,Z)+Yh(Z,X)-Zh(X,Y)\\
\label{eqLC1}
&+h([X,Y],Z)-h([X,Z],Y)-h([Y,Z],X).
\end{align}
From the definition of $h:=g^\psi$, we have
\begin{align}
\nonumber
&Xh(Y,Z)=Xg(\psi^{-1} Y,\psi^{-1} Z)\\
\nonumber
&=g(\nabla^g_X(\psi^{-1}Y),\psi^{-1}Z)+g(\psi^{-1}Y,\nabla^g_X(\psi^{-1}Z))\\
\nonumber
&=g((\nabla^g_X\psi^{-1})Y,\psi^{-1}Z)+g(\psi^{-1}\nabla^g_XY,\psi^{-1}Z)\\
\nonumber
&+g(\psi^{-1}Y,(\nabla^g_X\psi^{-1})Z)+g(\psi^{-1}Y,\psi^{-1}\nabla^g_XZ)\\
\nonumber
&=h(\psi(\nabla^g_X\psi^{-1})Y,Z)+h(\nabla^g_XY,Z)\\
\label{eqLC2}
&+h(Y,\psi(\nabla^g_X\psi^{-1})Z)+h(Y,\nabla^g_XZ).
\end{align}
Permuting $X$, $Y$, and $Z$ in (\ref{eqLC2}) gives
\begin{align}
\nonumber
Yh(Z,X)&=h(\psi(\nabla^g_Y\psi^{-1})Z,X)+h(\nabla^g_YZ,X)\\
\label{eqLC3}
&+h(Z,\psi(\nabla^g_Y\psi^{-1})X)+h(Z,\nabla^g_YX)
\end{align}
and
\begin{align}
\nonumber
Zh(X,Y)&=h(\psi(\nabla^g_Z\psi^{-1})X,Y)+h(\nabla^g_ZX,Y)\\
\label{eqLC4}
&+h(X,\psi(\nabla^g_Z\psi^{-1})Y)+h(X,\nabla^g_ZY).
\end{align}
Since $\nabla$ is torsion free, we also have
\begin{align}
\label{eqLC5}
h([X,Y],Z)&=h(\nabla^g_XY,Z)-h(\nabla^g_YX,Z),\\
\label{eqLC6}
h([X,Z],Y)&=h(\nabla^g_XZ,Y)-h(\nabla^g_ZX,Y),\\
\label{eqLC7}
h([Y,Z],X)&=h(\nabla^g_YZ,X)-h(\nabla^g_ZY,X).
\end{align}
Substituting (\ref{eqLC2})-(\ref{eqLC7}) into (\ref{eqLC1}) gives (\ref{eqLCA}).  Lastly, if 
$$
(\nabla^g_X\psi^{-1})Y=(\nabla^g_Y\psi^{-1})X,
$$ 
then (\ref{eqLCA}) reduces to 
\begin{align*}
h(\nabla^h_XY,Z)&=h(\nabla^g_XY,Z)+h(\psi(\nabla^g_X\psi^{-1})Y,Z)\\
&=h(\nabla^g_XY+\psi(\nabla^g_X\psi^{-1})Y,Z)\\
&=h(\psi \nabla^g_X(\psi^{-1}Y),Z),
\end{align*}
which implies (\ref{eqLCB}).
\end{proof}

\noindent Motivated by Proposition \ref{propLeviCivita}, we make the following definition:
\begin{definition}
\label{defgCodazzi}
A bundle automorphism $\psi:TM\stackrel{\sim}{\rightarrow} TM$ is called a \textit{$g$-Codazzi map} if $(\nabla^g_X\psi^{-1})Y=(\nabla^g_Y\psi^{-1})X$.
\end{definition}
\noindent We point out that formula (\ref{eqLCB}), which is a special case of Proposition \ref{propLeviCivita}, appeared originally in the work of Hicks \cite{Hicks1965}. 

In the literature, there are two different, but related notions of what a Codazzi tensor is.  In the Riemannian geometry viewpoint, a Codazzi tensor  \cite{Spivak1979,doCarmo1992,DerdzinskiShen1983,NomizuSmyth1969} on a Riemannian manifold $(M,g)$ is a symmetric $(0,2)$-tensor $A$ satisfying
\begin{equation}
\label{eqCodazziClassic}
(\nabla^g_X A)(Y,Z)=(\nabla^g_Y A)(X,Z).
\end{equation}
In the affine connection viewpoint \cite{DerdzinskiShen1983}, a Codazzi tensor for the pair $(M,\nabla)$ where $\nabla$ is a fixed affine connection on $M$ is an endomorphism $B\in \mathrm{End}(TM)$ satisfying
\begin{equation}
\label{eqCodazziAffine}
(\nabla_XB)Y=(\nabla_YB)X.
\end{equation}
From the affine connection viewpoint, a $g$-Codazzi map is an automorphism $\psi\in \mathrm{Aut}(TM)$ such that $\psi^{-1}$ is Codazzi tensor for ther pair $(M,\nabla^g)$.  Note that $\psi$ itself is not a Codazzi tensor for the pair $(M,\nabla^g)$ since
\begin{equation}
\label{eqPsiSwitch}
(\nabla^g_X\psi)Y=(\nabla^g_{\psi Y}\psi)(\psi^{-1}X).
\end{equation}
In the current paper, we will always use ``Codazzi tensor" in the Riemannian geometry sense.  After reconciling the definition of a $g$-Codazzi map with the affine connection viewpoint of a Codazzi tensor, the following result\footnote{The author arrived at Proposition \ref{propCodazzi} independently and became aware of the result by Hicks \cite{Hicks1965} only later through the paper of Derdzinski and Shen \cite{DerdzinskiShen1983}.} is attributed to Hicks \cite{Hicks1965}.  We give a proof of the result for the sake of keeping the current paper as self-contained as possible.  
\begin{proposition}[Hicks, \cite{Hicks1965}]
\label{propCodazzi}
Let $(M,g)$ be a Riemannian manifold.  Let $\mathcal{C}$ denote the set of non-degenerate Codazzi tensors on $M$ and let $\mathcal{S}$ denote the set of automorphisms of $TM$ which are both $g$-Codazzi and self-adjoint with respect to $g$.  Then there is a one-to-one correspondence between $\mathcal{C}$ and $\mathcal{S}$.
\end{proposition}
\begin{proof}
Let $A\in \mathcal{C}$ and let $\psi: TM\stackrel{\sim}{\rightarrow} TM$ be the unique bundle automorphism defined by
\begin{equation}
\label{eqCodazzi1}
A(X,Y)=g(\psi^{-1} X,Y).
\end{equation}
The symmetry of $A$ implies that $\psi^{-1}$ (and hence $\psi$) is self-adjoint:
\begin{align*}
g(\psi^{-1}  X,Y)=A(X,Y)=A(Y,X)=g(\psi^{-1}  Y,X)=g(X,\psi^{-1}  Y).
\end{align*}
To show that $\psi$ is a $g$-Codazzi map, we let a vector field $Z$ act on both sides of (\ref{eqCodazzi1}).  The left side is then
\begin{equation}
\label{eqCodazzi2}
ZA(X,Y)=(\nabla^g_Z A)(X,Y)+A(\nabla^g_Z X,Y)+A(X,\nabla^g_ZY).
\end{equation}
The right side is
\begin{align}
\nonumber
Zg(\psi^{-1} X,Y)&=g((\nabla^g_Z\psi^{-1})X,Y)+g(\psi^{-1} \nabla^g_Z X,Y)+g(\psi^{-1} X,\nabla^g_ZY)\\
\label{eqCodazzi3}
&=g((\nabla^g_Z\psi^{-1})X,Y)+A(\nabla^g_ZX,Y)+A(X,\nabla^g_ZY).
\end{align}
Equations (\ref{eqCodazzi2}) and (\ref{eqCodazzi3}) imply
\begin{equation}
\label{eqCodazzi4}
g((\nabla^g_Z\psi^{-1})X,Y)=(\nabla^g_Z A)(X,Y)=(\nabla^g_X A)(Z,Y)=g((\nabla^g_X\psi^{-1})Z,Y).
\end{equation}
Since $g$ is nondegenerate, we have $(\nabla^g_Z\psi^{-1})X=(\nabla^g_X\psi^{-1})Z$.  Hence, $\psi\in \mathcal{S}$. On the other hand, if we start with a $\psi\in \mathcal{S}$ and use (\ref{eqCodazzi1}) to define $A$, the same exact calculation shows that $A\in\mathcal{C}$.  Lastly, since everything is derived from (\ref{eqCodazzi1}), the maps $\mathcal{C}\rightarrow \mathcal{S}$ and $\mathcal{S}\rightarrow \mathcal{C}$ are clearly the inverse of one another.  This completes the proof.
\end{proof}

The simple transformation law of (\ref{eqLCB}) makes $g$-Codazzi maps a natural starting point for our study of $\psi$-twisted almost Hermitian structures.  For this reason, the current paper will focus almost exclusively on this particular class of automorphisms.  We regard $g$-Codazzi maps as an important stepping stone in our study of $\psi$-twisted almost Hermitian structures. 

Codazzi tensors arise in a number of different ways (see \cite{DerdzinskiShen1983} for examples).  One way (which is relevant to the current paper) is on Riemannian manifolds of constant sectional curvature.  Explicitly, if $(M,g)$ has constant sectional curvature $\kappa$, then for any smooth function $f$ on $M$ and any constant $c\in \mathbb{R}$, we obtain a Codazzi tensor $A_{f,c}$ with the following formula (cf \cite{Spivak1979,DerdzinskiShen1983}) :   
\begin{equation}
\label{eqCodazzi5}
A_{f,c}:=\mathrm{Hess}_gf + (\kappa f+c)g,
\end{equation}
where $(\mathrm{Hess}_gf)(X,Y):=(\nabla^g_Xdf)(Y)$ is the Hessian of $f$.  One can verify that $A_{f,c}$ is a Codazzi tensor by direct calculation and by recalling that on a manifold with constant sectional curvature $\kappa$, $R^g$ is given by the following simple formula:
$$
R^g(X,Y)Z = \kappa(g(Y,Z)X-g(X,Z)Y).
$$
Moreover, if $(M,g)$ is also simply connected, then every Codazzi tensor is of the form given by (\ref{eqCodazzi5}) (see e.g. \cite{Besse1987}).  Note that if $M$ is also compact, we can always choose $c$ to be a sufficiently large positive (or negative) number so that $A_{f,c}$ is either positive or negative definite (and hence nondgenerate).   Consequently, the $n$-dimensional round sphere $S^n$ has an abundance of $g$-Codazzi maps.  Of course, the case of interest in this paper is the round $6$-sphere due to its nearly K\"{a}hler structure.  Hence, there is no shortage of twists by $g$-Codazzi maps that one can apply to the standard almost complex structure on $S^6$. We will prove later in Theorem \ref{thmS6Twist} that the standard almost complex structure on $S^6$ remains non-integrable after twisting by any $g$-Codazzi map.   

From Proposition \ref{propCodazzi}, a self-adjoint $g$-Codazzi map is equivalent to a nondegenerate Codazzi tensor.   However, it is easy to give examples of $g$-Codazzi maps which are not self-adjoint.  
\begin{example}
\label{exCodazziCounter}
Let $(M,g,J,\omega)$ be any K\"{a}hler manifold.  Let $\psi= a\cdot \mbox{id}+bJ$ where $a,b$ are nonzero constants.  Then 
$$
\psi^{-1}=\frac{1}{a^2+b^2}(a\cdot \mbox{id}-bJ).
$$ 
Since $\nabla^g J=0$ for a K\"{a}hler manifold, we have $\nabla^g_X\psi^{-1}=0$.  In particular, $(\nabla^g_X\psi^{-1})Y=(\nabla^g_Y\psi^{-1})X$ which shows that $\psi$ is a $g$-Codazzi map. Moreover, $\psi$ is not self-adjoint: 
$$
\psi^\dagger=a\cdot \mbox{id}-bJ\neq \psi,\hspace*{0.1in}\forall~b\neq 0.
$$
\end{example}
\noindent For the round sphere $(S^n,g)$, we will  show later that every $g$-Codazzi map must be self-adjoint for $n\ge 3$ (see  Theorem \ref{thmgCodazziSnSA}).  This observation is most likely well known to experts.  However, the author has not been able to find an explicit reference for it.  The next result relates the curvature endomorphism $R^h$ to $R^g$ for $h:=g^\psi$.  This result\footnote{The author arrived at Proposition \ref{propCurvatureEndomorphism} independently and became aware of the result by Hicks \cite{Hicks1965} only later through the paper of Derdzinski and Shen \cite{DerdzinskiShen1983}.} was originally proved by Hicks \cite{Hicks1965}.  We include the proof for the sake of completeness.  
\begin{proposition}[Hicks, \cite{Hicks1965}]
\label{propCurvatureEndomorphism}
Let $\psi\in \mbox{Aut}(TM)$ be a $g$-Codazzi map.  Let $h:=g^\psi$.  Then $R^h(X,Y)=\psi\circ R^g(X,Y)\circ \psi^{-1}$.  In particular, $g$ is flat if and only if $g^\psi$ is flat.
\end{proposition}
\begin{proof}
By Proposition \ref{propLeviCivita}, we have 
\begin{align*}
\nabla^h_X\nabla^h_YZ=\psi\nabla^g_X\nabla^g_Y(\psi^{-1}Z)
\end{align*}
From this, we have
\begin{align*}
R^h(X,Y)Z&=\nabla^h_X\nabla^hYZ-\nabla^h_Y\nabla^h_X Z-\nabla^h_{[X,Y]}Z\\
&=\psi(\nabla^g_X\nabla^g_Y(\psi^{-1}Z)-\nabla^g_Y\nabla^g_X(\psi^{-1}Z)-\nabla^g_{[X,Y]}(\psi^{-1}Z))\\
&=\psi R^g(X,Y)(\psi^{-1}Z)\\
&=(\psi\circ R^g(X,Y)\circ \psi^{-1})Z.
\end{align*}
\end{proof}
\noindent The next result is a formula relating the curvature operator on 2-forms $\mathcal{R}^h$ to $\mathcal{R}^g$ for $h:=g^\psi$.  As far as the author can tell, this particular formula does not appear in the paper of Hicks \cite{Hicks1965},  Derdzinski and Shen \cite{DerdzinskiShen1983}, or anywhere else.  
\begin{proposition}
\label{propCurvatureOperator}
Let $\psi\in \mbox{Aut}(TM)$ be a $g$-Codazzi map and let $h:=g^\psi$.  Then $\mathcal{R}^h =\mathcal{R}^g\circ \psi^\ast$, where $\psi^\ast$ is just the pull back on 2-forms, that is, for $\beta\in \wedge^2 T^\ast M$, $\psi^\ast\beta:=\beta(\psi\cdot, \psi\cdot)$.
\end{proposition}
\begin{proof}
Let $p\in M$ and let $\beta,\gamma\in \wedge^2 T^\ast_pM$.  Without loss of generality, we may assume $\beta=\beta_1\wedge \beta_2$ and $\gamma=\gamma_1\wedge \gamma_2$ where $\beta_i,\gamma_j\in T^\ast_pM$. Using Proposition \ref{propCurvatureEndomorphism}, we have
\begin{align}
\nonumber
h^{-1}(\beta,\mathcal{R}^h(\gamma))&:=-\mbox{Rm}^h(\beta^{\sharp_h},\gamma^{\sharp_h})\\
\nonumber
&=-h(R^h(\beta_1^{\sharp_h},\beta_2^{\sharp_h})\gamma_1^{\sharp_h},\gamma_2^{\sharp_h})\\
\nonumber
&=-g(R^g(\beta_1^{\sharp_h},\beta_2^{\sharp_h})\psi^{-1}\gamma_1^{\sharp_h},\psi^{-1}\gamma_2^{\sharp_h})\\
\label{eqCO1}
&=-\mbox{Rm}^g(\beta^{\sharp_h},\psi^{-1}(\gamma^{\sharp_h}))
\end{align}
where $\psi^{-1}: \wedge^2 TM\rightarrow \wedge^2 TM$ is defined by
$$
\psi^{-1}(X_1\wedge X_2):=(\psi^{-1}X_1)\wedge (\psi^{-1}X_2)
$$  
for $X_1,X_2\in T_pM$.  We now express $\sharp_h$ in terms of $\sharp_g$.  This is just a matter of unpacking the definitions.   Let $Y\in T_pM$ and $\sigma\in T^\ast_pM$.  For convenience, write 
$$
\flat_g(Y):=Y^{\flat_g},\hspace*{0.1in} \flat_h(Y):=Y^{\flat_h}
$$
and
$$
\sharp_g(\sigma):=\sigma^{\sharp_g},\hspace*{0.1in}\sharp_h(\sigma):=\sigma^{\sharp_h}.
$$
Then
\begin{align}
\nonumber
\flat_h(Y)&=h(Y,\cdot)\\
\nonumber
&=g(\psi^{-1}Y,\psi^{-1}\cdot)\\
\nonumber
&=\flat_g(\psi^{-1}Y)\circ \psi^{-1}\\
\nonumber
&=(\psi^{-1})^\ast \left(\flat_g(\psi^{-1}Y)\right).
\end{align}
From this, we have
\begin{equation}
\label{eqFlatGH}
\flat_h=(\psi^{-1})^\ast \circ \flat_g\circ \psi^{-1}.
\end{equation}
Taking the inverse on both sides of (\ref{eqFlatGH}) gives
\begin{equation}
\label{eqSharpGH}
\sharp_h=\psi\circ \sharp_g\circ \psi^\ast.
\end{equation}
Applying (\ref{eqSharpGH}) to $\beta^{\sharp_h}$ gives
\begin{equation}
\label{eqSharpBeta}
\beta^{\sharp_h}=\psi (\psi^\ast\beta)^{\sharp_g}
\end{equation}
where (again) $\psi: \wedge^2TM\rightarrow \wedge^2 TM$ is defined by letting $\psi$ act on each factor of the wedge product, that is, 
$$
\psi(X_1\wedge X_2):=(\psi X_1)\wedge (\psi X_2),
$$  
for $X_1,X_2\in T_pM$.  Likewise, applying (\ref{eqSharpGH}) to $\psi^{-1}(\gamma^{\sharp_h})$ gives
\begin{equation}
\label{eqSharpGamma}
\psi^{-1}(\gamma^{\sharp_h})=(\psi^\ast \gamma)^{\sharp_g}.
\end{equation}
Substituting (\ref{eqSharpBeta}) and (\ref{eqSharpGamma}) into (\ref{eqCO1}) gives
\begin{align}
\label{eqCO2}
h^{-1}(\beta,\mathcal{R}^h(\gamma))&=-\mbox{Rm}^g(\psi (\psi^\ast\beta)^{\sharp_g},(\psi^\ast \gamma)^{\sharp_g}).
\end{align}
Let $\lambda\in \wedge^2T^\ast_pM$ be defined by
\begin{equation}
\label{eqCO2a}
\lambda^{\sharp_g}=\psi (\psi^\ast\beta)^{\sharp_g}.
\end{equation}
Then (\ref{eqCO2}) can be rewritten as
\begin{align}
\nonumber
h^{-1}(\beta,\mathcal{R}^h(\gamma))&=-\mbox{Rm}^g(\lambda^{\sharp_g},(\psi^\ast \gamma)^{\sharp_g})\\
\nonumber
&=g^{-1}(\lambda,\mathcal{R}^g(\psi^\ast \gamma))\\
\nonumber
&=\mathcal{R}^g(\psi^\ast \gamma)(\lambda^{\sharp_g})\\
\nonumber
&=\mathcal{R}^g(\psi^\ast \gamma)\left(\psi (\psi^\ast\beta)^{\sharp_g}\right)\\
\nonumber
&=\psi^\ast\left[\mathcal{R}^g(\psi^\ast \gamma)\right]\left((\psi^\ast\beta)^{\sharp_g}\right)\\
\label{eqCO3}
&=g^{-1}(\psi^\ast\beta,\psi^\ast\left[\mathcal{R}^g(\psi^\ast \gamma)\right]).
\end{align}
On the other hand, for $\alpha,\eta\in \wedge^kT^\ast_pM$, (\ref{eqSharpGH}) implies 
\begin{align}
\nonumber
h^{-1}(\alpha,\eta)&=h(\alpha^{\sharp_h},\eta^{\sharp_h})\\
\nonumber
&=h(\psi(\psi^\ast\alpha)^{\sharp_g},\psi(\psi^\ast\eta)^{\sharp_g})\\
\nonumber
&=g((\psi^\ast\alpha)^{\sharp_g},(\psi^\ast\eta)^{\sharp_g})\\
\label{eqCO4}
&=g^{-1}(\psi^\ast\alpha,\psi^\ast\eta).
\end{align}
Hence,
\begin{equation}
\label{eqCO5}
h^{-1}(\cdot,\cdot)=g^{-1}(\psi^\ast\cdot,\psi^\ast\cdot).
\end{equation}
Applying (\ref{eqCO5}) to (\ref{eqCO3}) gives
\begin{equation}
\label{eqCO6}
g^{-1}(\psi^\ast\beta,\psi^\ast\mathcal{R}^h(\gamma))=g^{-1}(\psi^\ast\beta,\psi^\ast\left[\mathcal{R}^g(\psi^\ast \gamma)\right]).
\end{equation}
Comparing both sides of (\ref{eqCO6}) and using the fact that $\psi$ is invertible, we conclude 
$$
\mathcal{R}^h(\gamma)=\mathcal{R}^g(\psi^\ast \gamma)
$$
which completes the proof.
\end{proof}
\noindent Proposition \ref{propCurvatureOperator} is essential in establishing that every $g$-Codazzi map on the round sphere $(S^n,g)$ for $n\ge 3$ is self-adjoint (see Theorem \ref{thmgCodazziSnSA}). The latter will, in turn, play an important role in the proof of  Theorem \ref{thmS6Twist}. We conclude this section with the following result and a simple application.
\begin{proposition}
\label{propOmegaPsi}
Let $(M,g,J,\omega)$ be an almost Hermitian manifold and let $\psi\in \mbox{Aut}(TM)$ be a $g$-Codazzi map.  Then
\begin{align*}
(d\omega^\psi)(X,Y,Z)&=(\nabla^g_X\omega)^\psi(Y,Z)+(\nabla^g_Y\omega)^\psi(Z,X)+(\nabla^g_Z\omega)^\psi(X,Y),
\end{align*}
where $(\nabla^g_X\omega)^\psi(Y,Z):=(\nabla^g_X\omega)(\psi^{-1}Y,\psi^{-1}Z)$.  Moreover, if $(M,g,J,\omega)$ is K\"{a}hler, then $(M,g^\psi,J^\psi,\omega^\psi)$ is also K\"{a}hler.
\end{proposition}

\begin{proof}
Let $\beta:=\omega^\psi$.  Since $\nabla^g$ is torsion free, we may express $d\beta$ as 
\begin{equation}
\label{eqOmegaPsi1}
(d\beta)(X,Y,Z)=(\nabla^g_X\beta)(Y,Z)+(\nabla^g_Y\beta)(Z,X)+(\nabla^g_Z\beta)(X,Y). 
\end{equation}
Expanding $(\nabla^g_X\beta)(Y,Z)$ gives
\begin{align}
\nonumber
(\nabla^g_X&\beta)(Y,Z)=X(\beta(Y,Z))-\beta(\nabla^g_XY,Z)-\beta(Y,\nabla^g_XZ)\\
\nonumber
&=X(\omega(\psi^{-1}Y,\psi^{-1}Z))-\omega(\psi^{-1}\nabla^g_XY,\psi^{-1}Z)-\omega(\psi^{-1}Y,\psi^{-1}\nabla^g_XZ)\\
\nonumber
&=(\nabla^g_X\omega)(\psi^{-1}Y,\psi^{-1}Z)+\omega((\nabla^g_X\psi^{-1})Y,\psi^{-1}Z)\\
\nonumber
&+\omega(\psi^{-1}\nabla^g_XY,\psi^{-1}Z)+\omega(\psi^{-1}Y,(\nabla^g_X\psi^{-1})Z)\\
\nonumber
&+\omega(\psi^{-1}Y,\psi^{-1}\nabla^g_XZ)-\omega(\psi^{-1}\nabla^g_XY,\psi^{-1}Z)\\
\nonumber
&-\omega(\psi^{-1}Y,\psi^{-1}\nabla^g_XZ)\\
\nonumber
&=(\nabla^g_X\omega)(\psi^{-1}Y,\psi^{-1}Z)+\omega((\nabla^g_X\psi^{-1})Y,\psi^{-1}Z)\\
\label{eqOmegaPsi2}
&+\omega(\psi^{-1}Y,(\nabla^g_X\psi^{-1})Z).
\end{align}
Similarly,
\begin{align}
\nonumber
(\nabla^g_Y\beta)(Z,X)&=(\nabla^g_Y\omega)(\psi^{-1}Z,\psi^{-1}X)+\omega((\nabla^g_Y\psi^{-1})Z,\psi^{-1}X)\\
\label{eqOmegaPsi3}
&+\omega(\psi^{-1}Z,(\nabla^g_Y\psi^{-1})X)
\end{align}
and
\begin{align}
\nonumber
(\nabla^g_Z\beta)(X,Y)&=(\nabla^g_Z\omega)(\psi^{-1}X,\psi^{-1}Y)+\omega((\nabla^g_Z\psi^{-1})X,\psi^{-1}Y)\\
\label{eqOmegaPsi4}
&+\omega(\psi^{-1}X,(\nabla^g_Z\psi^{-1})Y).
\end{align}
Substituting (\ref{eqOmegaPsi2})-(\ref{eqOmegaPsi4}) into (\ref{eqOmegaPsi1}) and applying the fact that $\psi$ is a $g$-Codazzi map and $\omega$ is skew-symmetric gives
\begin{align*}
(d\beta)(X,Y,Z)&=(\nabla^g_X\omega)(\psi^{-1}Y,\psi^{-1}Z)+(\nabla^g_Y\omega)(\psi^{-1}Z,\psi^{-1}X)\\
&+(\nabla^g_Z\omega)(\psi^{-1}X,\psi^{-1}Y).
\end{align*}
Lastly, recall that $(M,g,J,\omega)$ is K\"{a}hler if and only if $\nabla^g J=0$.  Let $h:=g^\psi$ and $I:=J^\psi$.  Proposition \ref{propLeviCivita} now implies 
\begin{align*}
(\nabla^h_X I)Y&=\nabla^h_X(I Y)-I\nabla^h_X Y\\
&=\psi\nabla^g_X(J\psi^{-1}Y)-\psi J\nabla^g_X(\psi^{-1}Y)\\
&=\psi\left(\nabla^g_XJ\right)(\psi^{-1}Y)\\
&=0.
\end{align*}
\end{proof}

\begin{example}
Let $M=\mathbb{C}^2=\mathbb{R}^4$ and let $(g,J,\omega)$ denote the natural K\"{a}hler structure on $M$.  Let $(x_1,x_2,x_3,x_4)$ be the natural coordinates on $M$.  Explicitly, $g$, $J$, and $\omega$ are given by
$$
g=\sum_{i=1}^4dx_i\otimes dx_i,\hspace*{0.1in}\omega = dx_1\wedge dx_2+dx_3\wedge dx_4,
$$
$$
J\partial_1=\partial_2,~J\partial_2=-\partial_1,~J\partial_3=\partial_4,~J\partial_4=-\partial_3.
$$
Now let $f=\sin(x_1)\sin(x_3)$.  The associated Codazzi tensor is then
$$
A_{f,c}=\mbox{Hess}_gf+cg.
$$
Let $\psi\in \mathrm{Aut}(TM)$ be the $g$-Codazzi map defined by
$$
A_{f,c}(X,Y)=g(\psi^{-1}X,Y).
$$
With $\varphi: TM\rightarrow TM$ defined by $(\mathrm{Hess}_gf)(X,Y)=g(\varphi X,Y)$, we have
$$
\psi^{-1}=\varphi+c\cdot id.
$$
The matrix representation of $\psi^{-1}$ with respect to the global frame $(\partial_1,\partial_2,\partial_3,\partial_4)$ is given by
$$
[\psi^{-1}]=\begin{pmatrix}
\alpha & 0 & \beta & 0\\
0 & c & 0 & 0\\
\beta & 0 & \alpha & 0\\
0 & 0 & 0 & c 
\end{pmatrix}
$$
where 
$$
\alpha:=c-\sin(x_1)\sin(x_3),\hspace*{0.2in}\beta:=\cos(x_1)\cos(x_3)
$$
and we understand $\psi^{-1}\partial_j=\sum_i[\psi^{-1}]_{ij}\partial_i$.   We impose the condition $c>1$ to ensure that $\psi$ is  invertible.  The matrix representation of $J^\psi:=\psi J\psi^{-1}$ is then
$$
[J^\psi]=\begin{pmatrix}
0 & -\frac{c\alpha}{D} & 0 & \frac{c\beta}{D}\\
\frac{\alpha}{c} & 0 & \frac{\beta}{c} & 0\\
0 & \frac{c\beta}{D} & 0 & -\frac{c\alpha}{D}\\
\frac{\beta}{c} & 0 & \frac{\alpha}{c} & 0
\end{pmatrix}
$$
where $D:=\alpha^2-\beta^2$.  The $\psi$-twisted metric is 
\begin{align*}
g^\psi&=(\alpha^2+\beta^2)(dx_1\otimes dx_1+dx_3\otimes dx_3)+2\alpha\beta(dx_1\otimes dx_3+dx_3\otimes dx_1)\\
&+c^2(dx_2\otimes dx_2+dx_4\otimes dx_4).
\end{align*}
The associated fundamental form is 
\begin{align*}
\omega^\psi&=c\alpha(dx_1\wedge dx_2+dx_3\wedge dx_4)+c\beta(dx_1\wedge dx_4-dx_2\wedge dx_3).
\end{align*}
By Proposition \ref{propOmegaPsi}, the $\psi$-twisted manifold $(\mathbb{R}^4,g^\psi,J^\psi,\omega^\psi)$ is still a K\"{a}hler manifold.   Moreover, by Proposition \ref{propCurvatureEndomorphism}, $g^\psi$ is a flat metric.
\end{example}

\section{$g$-Codazzi maps on the $n$-sphere}
\label{secgCodazziSn}
\noindent In this section, we study $g$-Codazzi maps on the round $n$-sphere where we ultimately show that every $g$-Codazzi map must be self-adjoint for $n\ge 3$.  We begin with the following general result:
\begin{proposition}
\label{propEtaA1}
Let $(M,g)$ be a Riemannian manifold and let $\psi$ be a $g$-Codazzi map.  Define $\eta(X,Y):=g(\psi^{-1} X,Y)$ and define $\eta_a\in \Omega^2(M)$ by
$$
\eta_a(X,Y):=\frac{1}{2}(\eta(X,Y)-\eta(Y,X)).
$$
Then $\eta_a$ is closed.
\end{proposition}
\begin{proof}
By direct calculation, one easily finds that 
$$
(\nabla^g_X\eta)(Y,Z)=g((\nabla^g_X\psi^{-1})Y,Z).
$$
From this, we have
$$
(\nabla^g_X\eta_a)(Y,Z)=\frac{1}{2}(g((\nabla^g_X\psi^{-1})Y,Z)-g((\nabla^g_X\psi^{-1})Z,Y)).
$$
Using the above expression together with the condition 
$$
(\nabla^g_X\psi^{-1})Y=(\nabla^g_Y\psi^{-1})X,
$$
one easily finds
$$
(d\eta_a)(X,Y,Z)=(\nabla^g_X\eta_a)(Y,Z)+(\nabla^g_Y\eta_a)(Z,X)+(\nabla^g_Z\eta_a)(X,Y)=0.
$$
\end{proof}
Proposition \ref{propEtaA1} shows that the 2-form $\eta_a$ associated to a $g$-Codazzi map is necessarily closed.  However, there does not appear to be any reason why $\eta_a$ must be coclosed.  As a curiosity, we express the condition that $\eta_a$ is coclosed in terms of its associated $g$-Codazzi map.
\begin{proposition}
\label{propEtaA3}
Let $(M,g)$ be an oriented Riemannian manifold and let $\delta:\Omega^{k+1}(M)\rightarrow \Omega^k(M)$ be the codifferential.  Suppose $\psi$ is a $g$-Codazzi map and let $\eta_a\in \Omega^2(M)$ be defined as in Proposition \ref{propEtaA1}.  Then $\delta\eta_a=0$ if and only if for any local orthonormal frame $\{e_1,\dots, e_n\}$ and vector field $X$ on $M$, the following condition is satisfied:
$$
\mathrm{Tr}(\nabla^g_X\psi^{-1})=\sum_jg((\nabla^g_{e_j}\psi^{-1})e_j,X),
$$ 
where both sides are evaluated over the domain of $\{e_1,\dots,e_n\}$.
\end{proposition}
\begin{proof}
From the proof of Proposition \ref{propEtaA1}, we have
$$
(\nabla^g_X\eta_a)(Y,Z)=\frac{1}{2}(g((\nabla^g_X\psi^{-1})Y,Z)-g((\nabla^g_X\psi^{-1})Z,Y)).
$$
By direct calculation, we have
\begin{align*}
(\delta\eta_a)(X)&=-\sum_j(\iota_{e_j} (\nabla^g_{e_j}\eta_a)(X)\\
&=-\sum_j(\nabla^g_{e_j}\eta_a)(e_j,X)\\
&=-\frac{1}{2}\sum_j(g((\nabla^g_{e_j}\psi^{-1})e_j,X)+\frac{1}{2}\sum_j g((\nabla^g_{e_j}\psi^{-1})(X,e_j))\\
&=-\frac{1}{2}\sum_j(g((\nabla^g_{e_j}\psi^{-1})e_j,X)+\frac{1}{2}\sum_j g((\nabla^g_{X}\psi^{-1})e_j,e_j))\\
&=-\frac{1}{2}\sum_j(g((\nabla^g_{e_j}\psi^{-1})e_j,X)+\frac{1}{2}\mathrm{Tr}(\nabla^g_X\psi^{-1}).
\end{align*}
Proposition \ref{propEtaA3} now follows immediately from this.
\end{proof}
\noindent We now turn our attention to the round sphere $S^n$.  For the rest of this section, we let $g$ denote the round metric on $S^n$ unless stated otherwise.  The next result rules out the possibility that the round sphere $S^n$ admits any skew-adjoint $g$-Codazzi maps for $n\ge 3$.
\begin{proposition}
\label{propEtaA2}
The round sphere $(S^n,g)$ for $n\ge 3$ admits no $g$-Codazzi maps which are skew-adjoint.
\end{proposition}
\begin{proof}
Suppose $\psi$ is a nonzero $g$-Codazzi map which is skew-adjoint.  Then $\psi^{-1}$ is also skew-adjoint.   Let $\eta(X,Y):=g(\psi^{-1} X,Y)$.  Since $\psi^{-1}$ is skew-adjoint, it follows immediately that $\eta$ is a 2-form and $\eta_a=\eta$, where $\eta_a$ is defined as in Proposition \ref{propEtaA1}.  From the proof of Proposition \ref{propEtaA1}, we have $(\nabla^g_X\eta)(Y,Z)=g((\nabla^g_X\psi^{-1})Y,Z)$.  However, since $\psi$ is a $g$-Codazzi map, we also have
$$
(\nabla^g_X\eta)(Y,Z)=g((\nabla^g_Y\psi^{-1})X,Z)=(\nabla^g_Y\eta)(X,Z).
$$
Expressing $d\eta$ in terms of $\nabla^g$ gives
\begin{align*}
(d\eta)(X,Y,Z)&=(\nabla^g_X\eta)(Y,Z)+(\nabla^g_Y\eta)(Z,X)+(\nabla^g_Z\eta)(X,Y)\\
&=(\nabla^g_X\eta)(Y,Z)-(\nabla^g_Y\eta)(X,Z)+(\nabla^g_X\eta)(Z,Y)\\
&=(\nabla^g_X\eta)(Y,Z)-(\nabla^g_X\eta)(Y,Z)-(\nabla^g_X\eta)(Y,Z)\\
&=-(\nabla^g_X\eta)(Y,Z).
\end{align*}
Since $\eta=\eta_a$ is closed by Proposition \ref{propEtaA1}, we see that $\eta$ is a parallel 2-form on $S^n$.  However, on the round sphere $S^n$, every parallel $k$-form for $0<k<n$ is precisely zero (see e.g., Ch. X of \cite{KobNomII}).  Since $n\ge 3$ and $\eta$ is a 2-form, we have $\eta=0$. This implies that $\psi^{-1}=0$.  However, this is a contradiction since, in particular, $\psi$ is an element of $\mathrm{Aut}(TS^n)$.
\end{proof}
\noindent Now we need a technical lemma which is pure linear algebra.
\begin{lemma}
\label{lemSA}
Let $V$ be a finite dimensional real vector space of dimension $n\ge 3$ with inner product $\mu$ and let $F: V\rightarrow V$ be a vector space isomorphism with the property that $F^\ast: \wedge^2 V^\ast\rightarrow \wedge^2 V^\ast$ is self-adjoint with respect to $\mu^{-1}$.  Then $F$ is either self-adjoint or skew-adjoint with respect to $\mu$.
\end{lemma}
\begin{proof}
Let $F^\ast: \wedge^k V^\ast\rightarrow \wedge^k V^\ast$ denote the usual pullback on $k$-forms.  Let $F_k: \wedge^k V\rightarrow \wedge^kV$ be the vector space isomorphism which is uniquely defined by
$$
F_k(v_1\wedge \cdots \wedge v_k):=(Fv_1)\wedge \cdots \wedge (Fv_k),
$$
for $v_1,\dots, v_k\in V$.  By direct calculation, one easily finds
\begin{equation}
\label{eqLemmaSA1}
F^\ast (v^{\flat})= (F^\dagger v)^{\flat},
\end{equation}
where $v\in V$, $F^\dagger$ is the adjoint of $F$, and the $\flat$-operation is defined with respect to $\mu$.  For $X\in \wedge^k V$, equation (\ref{eqLemmaSA1}) and the definition of $(F^\dagger)_k$ implies
\begin{equation}
\label{eqLemmaSA2}
F^\ast(X^\flat)=\left[ (F^\dagger)_k(X)\right]^\flat.
\end{equation}
Since $F^\ast: \wedge^2 V^\ast\rightarrow\wedge^2 V^\ast$ is self-adjoint with respect to $\mu^{-1}$, equation (\ref{eqLemmaSA2}) implies
\begin{equation}
\label{eqLemmaSA3}
\mu^{-1}\left(\left[(F^\dagger)_2(X)\right]^\flat,Y^\flat  \right)=\mu^{-1}\left( X^\flat, \left[(F^\dagger)_2(Y)\right]^\flat\right),
\end{equation}
for all $X,Y\in \wedge^2 V$.  From the definition of $\mu^{-1}$, equation (\ref{eqLemmaSA3}) can be rewritten as
\begin{equation}
\label{eqLemmaSA4}
\mu\left((F^\dagger)_2(X),Y\right)=\mu\left( X, (F^\dagger)_2(Y)\right).
\end{equation}
In other words, $(F^\dagger)_2: \wedge^2 V\rightarrow \wedge^2 V$ is self-adjoint with respect to $g$.  From the definition of $g$ applied to $\wedge^k V$, one always has the identity 
\begin{equation}
\label{eqLemmaSA4a}
(F_k)^\dagger = (F^\dagger)_k,\hspace*{0.1in}\forall~k.
\end{equation}
Equation (\ref{eqLemmaSA4a}) together with the self-adjointness of $(F^\dagger)_2$ implies that $F_2$ is also self-adjoint with respect to $\mu$:
\begin{equation}
\label{eqLemmaSA5}
(F_2)^\dagger=F_2.
\end{equation}
Now let $e_1,\dots, e_n$ be an orthonormal basis of $V$ with respect to $\mu$ and write $Fe_j=\sum_i F_{ij}e_i$.  Expanding
$$
\mu(F_2(e_i\wedge e_j),e_k\wedge e_l)=\mu(e_i\wedge e_j,F_2(e_k\wedge e_l))
$$
gives
\begin{equation}
\label{eqlemSA1}
F_{ki}F_{lj}-F_{li}F_{kj}=F_{ik}F_{jl}-F_{jk}F_{il}.
\end{equation}
For convenience, identify $F$ and $F^{-1}$ with their matrix representations with respect to $e_1,\dots, e_n$ where the $(i,j)$-element of $F^{-1}$ is denoted as $F^{ij}$.   Also, let $H=(H_{ij})=F^TF^{-1}$.  
Multiplying both sides of (\ref{eqlemSA1}) by $F^{kp}F^{lq}$ and then summing over all $k$ and $l$ gives
\begin{equation}
\label{eqlemSA2}
H_{ip}H_{jq}-H_{iq}H_{jp}=\delta_i^p\delta_j^q-\delta_j^p\delta_i^q.
\end{equation}
We can rewrite (\ref{eqlemSA2}) as 
\begin{equation}
\label{eqlemSA2a}
\det\begin{pmatrix}
H_{ip} & H_{iq}\\
H_{jp} & H_{jq}
\end{pmatrix}=\delta_i^p\delta_j^q-\delta_j^p\delta_i^q.
\end{equation}
For $i\neq j$, equation (\ref{eqlemSA2a}) gives
\begin{equation}
\label{eqlemSA3}
\det \begin{pmatrix}
H_{ii} & H_{ij}\\
H_{ji} & H_{jj}
\end{pmatrix}=1.
\end{equation}
Since $n\ge 3$, equation (\ref{eqlemSA2}) implies that for all $i,j,p$ distinct, we have 
\begin{equation}
\label{eqlemSA4}
\det \begin{pmatrix}
H_{ii} & H_{ip}\\
H_{ji} & H_{jp}
\end{pmatrix}=\det \begin{pmatrix}
H_{ip} & H_{ij}\\
H_{jp} & H_{jj}
\end{pmatrix}=0.
\end{equation}
From (\ref{eqlemSA3}), we have
$$
\begin{pmatrix}
H_{ii}\\
H_{ji}
\end{pmatrix} \neq 0,\hspace*{0.1in} \begin{pmatrix}
H_{ij}\\
H_{jj}
\end{pmatrix}\neq 0.
$$
This together with (\ref{eqlemSA4}) implies
\begin{equation}
\label{eqlemSA5}
\begin{pmatrix}
H_{ip}\\
H_{jp}
\end{pmatrix}=\lambda_1\begin{pmatrix}
H_{ii}\\
H_{ji}
\end{pmatrix}=\lambda_2\begin{pmatrix}
H_{ij}\\
H_{jj}
\end{pmatrix}
\end{equation}
for some $\lambda_1,\lambda_2\in \mathbb{R}$.  However, if $\lambda_1$ or $\lambda_2$ is nonzero, then the determinant appearing in (\ref{eqlemSA3}) would be $0$, which is a contradiction.  Hence, $\lambda_1=\lambda_2=0$.  Since $i,j,p$ are arbitrary distinct elements, it follows that $H_{ij}=0$ for all $i\neq j$.  In other words, $H$ is a diagonal matrix.  Equation (\ref{eqlemSA3}) now implies $H_{ii}H_{jj}=1$ for $i\neq j$.  From this, we have 
$$
H_{jj}=\frac{1}{H_{11}},\hspace*{0.1in} j=2,\dots, n.
$$
Now let $j,k$ be distinct with $j,k>1$ (this is possible since $n\ge 3$).  Then
$$
H_{jj}H_{kk}=\frac{1}{H_{11}^2}=1,
$$
which implies that $H_{11}=\pm 1$.  Since $H_{jj}=1/H_{11}$ for $j\ge 2$, it follows that $H=F^TF^{-1}=\pm I_n$.  Hence, $F^T=\pm F$.  Since the matrix representation of $F$ is taken with respect to an orthonormal basis $e_1,\dots,e_n$, it follows immediately that $F:V\rightarrow V$ is either self-adjoint or skew-adjoint.
\end{proof}
\noindent We now have everything we need to prove the following result:
\begin{theorem}
\label{thmgCodazziSnSA}
Let $(S^n,g)$ be the $n$-dimensional round sphere with $n\ge 3$.   Let $\psi\in \mbox{Aut}(TS^n)$ be a $g$-Codazzi map.  Then $\psi$ is self-adjoint with respect to $g$.
\end{theorem}
\begin{proof}
For $(S^n,g)$, the Riemann curvature tensor is
$$
\mbox{Rm}^g(X,Y,Z,W)=g(Y,Z)g(X,W)-g(X,Z)g(Y,W).
$$
Let $e_1,\dots, e_n$ be a local orthonormal frame and let $\theta_1,\dots,\theta_n$ be the dual frame.  Then for $i<j$ and $k<l$, we have
\begin{align*}
g^{-1}\left(\theta_i\wedge \theta_j,\mathcal{R}^g(\theta_k\wedge\theta_l)\right)&=-\mbox{Rm}^g(e_i\wedge e_j,e_k\wedge e_l)=\delta_{ik}\delta_{jl}.
\end{align*}
From this, it immediately follows that $\mathcal{R}^g=\mbox{id}$ on $\wedge^2 T^\ast S^n$.  Let $h:=g^\psi$.  By Proposition \ref{propCurvatureOperator}, we have
$$
\mathcal{R}^h=\mathcal{R}^g\circ \psi^\ast=\psi^\ast,
$$
which implies that $\psi^\ast$ is self-adjoint with respect to $h^{-1}$.  Applying Lemma \ref{lemSA} (pointwise) and using the fact that $S^n$ is connected, it follows that $\psi$ is either self-adjoint or skew-adjoint with respect to $h$.  However, from the definition of $h$, this means
\begin{align*}
&h(\psi X, Y)=\pm h(X,\psi Y)\\
&g(X,\psi^{-1}Y)=\pm g(\psi^{-1}X,Y),
\end{align*}
which shows that $\psi^{-1}$ (and hence $\psi$) is either self-adjoint or skew-adjoint with respect to $g$.  Since $n\ge 3$, Proposition \ref{propEtaA2} implies that $\psi$ cannot be skew-adjoint with respect to $g$.  Hence, $\psi$ must be self-adjoint with respect to $g$.  
\end{proof}
\begin{corollary}
\label{corCodazziIdentitySn}
Let $\psi$ be any $g$-Codazzi map on the round sphere $(S^n,g)$ for $n\ge 3$.  Then for any local orthonormal frame $e_1,\dots,e_n$ and any vector field $X$ on $S^n$, $\psi$ satisfies the following condition: 
$$
\mathrm{Tr}(\nabla^g_X\psi^{-1})=\sum_jg((\nabla^g_{e_j}\psi^{-1})e_j,X),
$$ 
where both sides are evaluated over the domain of $\{e_1,\dots,e_n\}$.
\end{corollary}
\begin{proof}
By Theorem \ref{thmgCodazziSnSA}, $\psi$ is self-adjoint with respect to $g$.  This implies that $\eta_a=0$ in Proposition \ref{propEtaA1}.  In particular, $\eta_a$ is (trivially) coclosed.  Corollary \ref{corCodazziIdentitySn} now follows from Proposition \ref{propEtaA3}.  
\end{proof}

\section{Integrability and $g$-Codazzi maps}
\label{SectionIntSixSphere}
In this section $(M,g,J,\omega)$ denotes a fixed almost Hermitian manifold until stated otherwise.   For $\psi\in \mbox{Aut}(TM)$, let
\begin{equation}
\rho^\psi(X,Y):=\psi^{-1}[\psi X,\psi Y]-[X,Y].
\end{equation}
The following result shows that the Nijhenhuis tensor of $J^\psi$ differs from that of $J$ by an expression which depends on $\rho^\psi$.  
\begin{proposition}
\label{propPsiIntegrable}
For any $\psi\in \mathrm{Aut}(TM)$, let 
$$
S^\psi(X,Y):=J\rho^\psi(JX,Y)+J\rho^\psi(X,JY)+\rho^\psi(X,Y)-\rho^\psi(JX,JY).
$$
Then 
\begin{align}
\label{eqPsiIntegrableA}
\psi^{-1} N_{J^\psi}(X,Y)=N_J(U,V)+S^\psi(U,V),
\end{align}
where $U:=\psi^{-1} X$ and $V:=\psi^{-1}Y$. In particular, $J^\psi$ is integrable if and only if
\begin{equation}
\label{eqPsiIntegrableB}
S^\psi(X,Y)=-N_J(X,Y).
\end{equation}
\end{proposition}
\begin{proof}
Let $I:=J^\psi$, $U:=\psi^{-1} X$, and $V:=\psi^{-1}Y$.  The proposition follows by direct calculation:
\begin{align*}
N_{I}(X,Y)&=I[IX,Y]+I[X,IY]+[X,Y]-[IX,IY]\\
&=I[\psi J U,\psi V]+I[\psi U,\psi J V]+[\psi U,\psi V]-[\psi J U,\psi J V]\\
&=\psi J(\psi^{-1}[\psi J U,\psi V]-[JU,V])+\psi J[JU,V]\\
&+\psi J (\psi^{-1}[\psi U, \psi JV]-[U,JV])+\psi J[U,JV]\\
&+\psi(\psi^{-1}[\psi U,\psi V]-[U,V])+\psi [U,V]\\
&-\psi(\psi^{-1}[\psi JU,\psi JV]-[JU,JV])-\psi [JU,JV]\\
&=\psi (J\rho^\psi(JU,V)+J\rho^\psi(U,JV)+\rho^\psi(U,V)-\rho^\psi(JU,JV))\\
&+\psi N_J(U,V)\\
&=\psi S^\psi(U,V)+\psi N_J(U,V).
\end{align*}
\end{proof}
\noindent Applying Proposition \ref{propPsiIntegrable} to the special case of $g$-Codazzi maps gives the following:
\begin{proposition}
\label{propIntCodazzi}
Let $\psi\in \mbox{Aut}(TM)$ be a $g$-Codazzi map.  Then $J^\psi$ is integrable if and only if
$$
J(\nabla^g_{\psi X}J)Y-J(\nabla^g_{\psi Y}J)X=(\nabla^g_{\psi JX}J)Y-(\nabla^g_{\psi JY}J)X.
$$
\end{proposition}
\begin{proof}
With the help of (\ref{eqPsiSwitch}), we can rewrite $\rho^\psi$ as 
\begin{align}
\nonumber
\rho^\psi(X,Y)&=\psi^{-1}[\psi X, \psi Y]-[X,Y]\\
\nonumber
&=\psi^{-1}\left( \nabla^g_{\psi X}(\psi Y)-\nabla^g_{\psi Y}(\psi X)\right)-[X,Y]\\
\nonumber
&=\psi^{-1}\left((\nabla^g_{\psi X}\psi)Y+\psi\nabla^g_{\psi X}Y\right)-\psi^{-1}\left((\nabla^g_{\psi Y}\psi)X+\psi\nabla^g_{\psi Y}X\right)\\
\nonumber
&-[X,Y]\\
\nonumber
&=\psi^{-1}\left((\nabla^g_{\psi Y}\psi)X+\psi\nabla^g_{\psi X}Y\right)-\psi^{-1}\left((\nabla^g_{\psi Y}\psi)X+\psi\nabla^g_{\psi Y}X\right)\\
\nonumber
&-[X,Y]\\
\label{eqCodazziInt1}
&=\nabla^g_{\psi X}Y-\nabla^g_{\psi Y}X-[X,Y].
\end{align}
Let $S^\psi$ be defined as in Proposition \ref{propPsiIntegrable}.  Using (\ref{eqCodazziInt1}), $S^\psi$ can be rewritten as
\begin{align}
\nonumber
S^\psi(X,Y)&=J\nabla^g_{\psi JX}Y-J\nabla^g_{\psi Y}(JX)-J[JX,Y]\\
\nonumber
&+J\nabla^g_{\psi X}(JY)-J\nabla^g_{\psi JY}X-J[X,JY]\\
\nonumber
&+\nabla^g_{\psi X}Y-\nabla^g_{\psi Y}X-[X,Y]\\
\nonumber
&-\nabla^g_{\psi JX}(JY)+\nabla^g_{\psi JY}(JX)+[JX,JY]\\
\nonumber
&=  J(\nabla^g_{\psi X}J)Y-J(\nabla^g_{\psi Y}J)X-(\nabla^g_{\psi JX}J)Y+(\nabla^g_{\psi JY}J)X\\
\label{eqCodzziInt2}
&-N_J(X,Y).
\end{align}
It now follows from Proposition \ref{propPsiIntegrable} that $J^\psi$ is integrable if and only if
$$
J(\nabla^g_{\psi X}J)Y-J(\nabla^g_{\psi Y}J)X=(\nabla^g_{\psi JX}J)Y-(\nabla^g_{\psi JY}J)X.
$$
\end{proof}
\noindent Applying Proposition \ref{propIntCodazzi} to nearly K\"{a}hler manifolds, one obtains
\begin{corollary}
\label{corNKCodazzi}
Let $(M,g,J,\omega)$ be a nearly K\"{a}hler manifold and let $\psi$ be a $g$-Codazzi map.  Then $J^\psi$ is integrable if and only if
\begin{equation}
\label{eqNKCodazzi1}
(\nabla^g_XJ)(J\psi+\psi J)Y=(\nabla^g_Y J)(J\psi +\psi J)X.
\end{equation}
In particular, if $(M,g,J,\omega)$ is strictly nearly K\"{a}hler, then no $g$-Codazzi map anti-commutes with $J$.
\end{corollary}
\begin{proof}
Recall that $(M,g,J,\omega)$ is nearly K\"{a}hler if $(g,J,\omega)$ is an almost Hermitian structure on $M$ such that $(\nabla^g_XJ)Y=-(\nabla^g_YJ)X$ for all $X,Y\in \mathfrak{X}(M)$.  $(M,g,J,\omega)$ is said to be strictly nearly K\"{a}hler if $J$ is also non-integrable.   With $(M,g,J,\omega)$ a nearly K\"{a}hler manifold and $\psi$ a $g$-Codazzi map, Proposition \ref{propIntCodazzi} implies that $J^\psi$ is integrable if and only if
\begin{align}
-J(\nabla^g_YJ)(\psi X)+J(\nabla^g_XJ)(\psi Y)=-(\nabla^g_YJ)(\psi JX)+(\nabla^g_XJ)(\psi JY).
\end{align}
Since  $J^2=-\mbox{id}$, it follows that $J(\nabla^g_X J)=-(\nabla^g_X J)J$ and the above equation can be rewritten as
$$
(\nabla^g_YJ)(J\psi X)-(\nabla^g_XJ)(J\psi Y)=-(\nabla^g_YJ)(\psi JX)+(\nabla^g_XJ)(\psi JY).
$$
By rearranging the terms, one obtains the desired identity.  

Lastly, suppose  $(M,g,J,\omega)$ is strictly nearly K\"{a}hler and $\psi$ is a $g$-Codazzi map which anti-commutes with $J$.  Then (\ref{eqNKCodazzi1}) is satisfied which in turn implies that $J^\psi$ is integrable.  However, this is a contradiction since $J^\psi=-J$.
\end{proof}
We recall that for a nearly K\"{a}hler manifold $(M,g,J,\omega)$ one has the following identity \cite{Gray1970}:
\begin{equation}
\label{eqNK1}
\frac{1}{3}d\omega(X,Y,Z)=(\nabla^g_X\omega)(Y,Z)=g((\nabla^g_XJ)Y,Z),
\end{equation}
where we note that the second equality holds for any almost Hermitian manifold.  The above identity implies 
\begin{equation}
\label{eqNK2}
d\omega(X,JY,JZ)=-d\omega(X,Y,Z),
\end{equation}
which in turn implies
\begin{equation}
\label{eqNK3}
d\omega(JX,Y,Z)=d\omega(X,JY,Z)=d\omega(X,Y,JZ).
\end{equation}
Using (\ref{eqNK1}), (\ref{eqNK3}) can be rewritten as
\begin{equation}
\label{eqNK4}
(\nabla^g_{JX}\omega)(Y,Z)=(\nabla^g_X\omega)(JY,Z)=(\nabla^g_X\omega)(Y,JZ).
\end{equation} 
Moreover, we also recall that $d\omega$ is of type $(3,0)+(0,3)$ for a nearly K\"{a}hler manifold.  
\begin{proposition}
\label{propNK2}
Let $(M,g,J,\omega)$ be nearly K\"{a}hler and suppose $\psi\in \mbox{Aut}(TM)$ is a $g$-Codazzi map such that $I:=J^\psi$ is integrable.  Let $F:=\psi^{-1}$.  Then 
\begin{align}
\nonumber
(\nabla^g_{Y}\omega)(FX,JFZ)&+(\nabla^g_{IY}\omega)(FX,FZ)\\
\label{eqNK2A}
&=(\nabla^g_{X}\omega)(FY,JFZ)+(\nabla^g_{IX}\omega)(FY,FZ),
\end{align}
for all vector fields $X,Y,Z$.  In particular, if $X,Y,Z$ are all $(1,0)$-vector fields or all $(0,1)$-vector fields, then
\begin{equation}
\label{eqNK2B}
(\nabla^g_X\omega)(FY,FZ)+(\nabla^g_Y\omega)(FZ,FX)=0.
\end{equation}
\end{proposition}

\begin{proof}
Applying the metric $g$ to both sides of (\ref{eqNKCodazzi1}) gives
\begin{align}
\nonumber
g((\nabla^g_XJ)&(J\psi Y),Z)+g((\nabla^g_XJ)(\psi J Y),Z)\\
\label{eqNK2a}
&=g((\nabla^g_Y J)(J\psi X),Z)+g((\nabla^g_Y J)((\psi J X),Z).
\end{align}
With the help of (\ref{eqNK1}), (\ref{eqNK2a}) can be rewritten as
\begin{align}
\nonumber
(\nabla^g_X\omega)&(J\psi Y,Z)+(\nabla^g_X\omega)(\psi JY,Z)\\
\label{eqNK2b}
&=(\nabla^g_Y\omega)(J\psi X,Z)+(\nabla^g_Y\omega)(\psi JX, Z).
\end{align}
Making the substitutions $X\rightarrow FX$, $Y\rightarrow FY$, and $Z\rightarrow FZ$ and using (\ref{eqNK4}), we obtain
\begin{align}
\nonumber
(\nabla^g_{FX}\omega)(Y,JFZ)&+(\nabla^g_{FX}\omega)(IY,FZ)\\
\label{eqNK2c}
&=(\nabla^g_{FY}\omega)(X,JFZ)+(\nabla^g_{FY}\omega)(IX,FZ).
\end{align}
Applying the total skew-symmetry of $(\nabla^g_X\omega)(Y,Z)$ to (\ref{eqNK2c}) gives (\ref{eqNK2A}).
From the definition of $I$, it follows that if $X$ is $(1,0)$ (resp. $(0,1)$) with respect to $I$, then $FX$ is $(1,0)$ (resp. $(0,1)$) with respect to $J$.  Hence if $X,Y,Z$  are all $(1,0)$ with respect to $I$, then (\ref{eqNK2A}) reduces to 
$$
2i(\nabla^g_Y\omega)(FX,FZ)=2i(\nabla^g_X\omega)(FY,FZ).
$$
Replacing $X,Y,Z$ with $(0,1)$-vector fields gives the same expression except $2i$ is now replaced with $-2i$.  This proves (\ref{eqNK2B}).
\end{proof}

\begin{remark}
As a consistency check, note that if we sum all cyclic permutations of (\ref{eqNK2B}) in Proposition \ref{propNK2}, we obtain 
$$
(\nabla^g_X\omega)(FY,FZ)+(\nabla^g_Y\omega)(FZ,FX)+(\nabla^g_Z\omega)(FX,FY)=0,
$$
where $F:=\psi^{-1}$ and $X,Y,Z$ are all $(1,0)$-vector fields or all $(0,1)$-vector fields.  By Proposition \ref{propOmegaPsi}, this means
$$
(d\omega^\psi)(X,Y,Z)=0
$$
whenever $X,Y,Z$ are all $(1,0)$-vector fields or all $(0,1)$-vector fields.  In other words, $d\omega^\psi$ is of type $(2,1)+(1,2)$.  Since $\omega^\psi$ is of type $(1,1)$ with respect to $I:=J^\psi$ and  $I$ is assumed to be integrable in the statement of Proposition \ref{propNK2}, the aforementioned condition on $d\omega^\psi$ is precisely what must hold under these conditions.  This observation serves as a nice consistency check on our formulas.   It is interesting to note that Proposition \ref{propNK2} does not seem to place any additional constraints on $d\omega^\psi$ that we can see.
\end{remark}

\begin{proposition}
\label{propNKCodazziSA}
Let $(M,g,J,\omega)$ be a nearly K\"{a}hler manifold and let $\psi\in \mbox{Aut}(TM)$ be a self-adjoint $g$-Codazzi map.  Then $J^\psi$ is integrable if and only if
$$
(\nabla^g_Z J)\circ \mathcal{K}-\mathcal{K}\circ (\nabla^g_ZJ)=0
$$
for all $Z\in \mathfrak{X}(M)$ where $\mathcal{K}:=J\psi+\psi J$.
\end{proposition}
\begin{proof}
By Corollary \ref{corNKCodazzi}, $J^\psi$ is integrable if and only if
\begin{equation}
\label{eqNKCodazziSA1}
g((\nabla^g_XJ)(\mathcal{K}Y),Z)=g((\nabla^g_YJ)(\mathcal{K}X),Z)
\end{equation}
for all $X,Y,Z\in \mathfrak{X}(M)$.  Since $\psi$ is self-adjoint and $J$ is skew-adjoint (with respect to $g$), it follows that $\mathcal{K}^\dagger=-\mathcal{K}$.  In addition, $\nabla^g_ZJ$ is also skew-adjoint with respect to $g$.  Indeed, this follows from the fact that, for any almost Hermitian manifold, one has
\begin{equation}
\label{eqNKCodazziSA2}
g((\nabla^g_ZJ)X,Y)=(\nabla^g_Z\omega)(X,Y).
\end{equation}
The nearly K\"{a}hler hypothesis now implies that the right side (and hence the left side) of the above equation is totally skew-symmetric in $X,Y,Z$.  Consequently, we have
\begin{equation}
\label{eqNKCodazziSA3}
g((\nabla^g_ZJ)X,Y)=-g((\nabla^g_ZJ)Y,X)=-g(X,(\nabla^g_ZJ)Y).
\end{equation}
In other words, $(\nabla^g_ZJ)^\dagger =-\nabla^g_ZJ$.  From this, (\ref{eqNKCodazziSA1}) can be rewritten as
\begin{align*}
&g((\nabla^g_XJ)(\mathcal{K}Y),Z)=g((\nabla^g_YJ)(\mathcal{K}X),Z),\\
&-g((\nabla^g_ZJ)(\mathcal{K}Y),X)=-g((\nabla^g_ZJ)(\mathcal{K}X),Y),\\
&g(\mathcal{K}Y,(\nabla^g_ZJ)X)=-g((\nabla^g_ZJ)(\mathcal{K}X),Y),\\
&-g(Y,\mathcal{K}(\nabla^g_ZJ)X)=-g((\nabla^g_ZJ)(\mathcal{K}X),Y).
\end{align*}
The last equality can be rewritten as
$$
g\left([(\nabla^g_ZJ)\circ \mathcal{K}-\mathcal{K}\circ(\nabla^g_ZJ)]X,Y\right)=0.
$$
Since $g$ is nondegenerate and $X,Y,Z$ are arbitrary vector fields, the desired equality follows.
\end{proof}
 
\begin{theorem}
\label{thmS6Twist}
Let $(S^6,g,J,\omega)$ be the standard nearly K\"{a}hler structure on the (unit) 6-sphere and let $\psi\in \mbox{Aut}(TM)$ be any $g$-Codazzi map.  Then $J^\psi$ is nonintegrable.
\end{theorem}
\begin{proof}
Since $(g,J,\omega)$ is a nearly K\"{a}hler structure, one has
\begin{equation}
\label{eqS6Twist1}
(\nabla^g_XJ)(JY)=-\frac{1}{4}N_J(X,Y),
\end{equation}
where the Nijenhuis tensor $N_J$ is (again) defined using the convention
$$
N_J(X,Y)=J[JX,Y]+J[X,JY]+[X,Y]-[JX,JY].
$$

Now let $\psi\in \mathrm{Aut}(TS^6)$ be any $g$-Codazzi map and let us suppose that $J^\psi$ is integrable.  It is well known that the Nijhenuis tensor of $J$ is ``nondegenerate" (cf \cite{Verbitsky2008}) in the sense that, for all $p\in S^6$, the linear map
$$
N_J:\wedge^2 T^{1,0}_pS^6\rightarrow T^{0,1}_pS^6, \hspace*{0.1in} X^+\wedge Y^+\mapsto N_J(X^+,Y^+)
$$
is an isomorphism where $X^+:=X-iJX$ for $X\in T_pS^6$.  By direct calculation, this is equivalent to the statement that for all $p\in S^6$ and all nonzero $X\in T_pS^6$, the kernel of the linear map from $T_pS^6$ to itself given by
$$
Y\mapsto N_J(X,Y),\hspace*{0.1in}\forall Y\in T_pS^6
$$
has kernel $U_X:=\mathrm{span}\{X,JX\}$.   The diligent reader can verify this by noting that the Lie group $G_2\subset SO(7)$ acts transitively on $S^6$ while preserving the standard nearly K\"{a}hler structure \cite{AgricolaBorowkaFriedrich2017}.  Hence, one only has to verify that the nondegeneracy condition of $N_J$ is satisfied at a single (convenient) point (e.g. $p=(0,0,0,0,0,0,1)$).

From (\ref{eqS6Twist1}), it now follows that for any nonzero $Z\in T_pS^6$, the linear map
$$
\nabla^g_ZJ:T_pS^6\rightarrow T_pS^6
$$
has kernel  $U_Z:=\mbox{span}\{Z,JZ\}$.  Moreover, as shown in the proof of Proposition \ref{propNKCodazziSA}, $\nabla^g_ZJ$ is skew-adjoint with respect to $g$:
$$
g((\nabla^g_ZJ)X,Y)=-g(X,(\nabla^g_ZJ)Y).
$$
From this, it follows that the image of $\nabla^g_ZJ$ lies inside $U_Z^\perp$ where $U_Z^\perp \subset T_pS^6$ is the $g$-orthogonal complement of $U_Z$.  In addition, since $T_pS^6=U_Z\oplus U_Z^\perp$ and the kernel of $\nabla^g_ZJ$ is $U_Z$, we also have $(\nabla^g_ZJ)(U_Z^\perp)=U_Z^\perp$.

By Theorem \ref{thmgCodazziSnSA}, $\psi$ is necessarily self-adjoint (with respect to $g$).  Since $\psi$ is a self-adjoint $g$-Codazzi map and $J^\psi$ is integrable (by our assumption), Proposition \ref{propNKCodazziSA} implies
\begin{equation}
\label{eqS6thm1}
(\nabla^g_ZJ)\circ \mathcal{K} =\mathcal{K}\circ (\nabla^g_ZJ),\hspace*{0.1in}\forall~p\in S^6,~Z\in T_pS^6,
\end{equation}
where $\mathcal{K}:=J\psi+\psi J$.  Since the kernel of $\nabla^g_ZJ$ is $U_Z$ and $(\nabla^g_ZJ)(U_Z^\perp)=U_Z^\perp$, it follows that $\mathcal{K}(U_Z)\subset U_Z$.  With $\psi$ self-adjoint and $J$ skew-adjoint, it follows readily that $\mathcal{K}$ is skew-adjoint.  Since $\{Z,JZ\}$ is a $g$-orthogonal basis, it follows that the matrix representation of $\mathcal{K}|_{U_Z}$ with respect to $\{Z,JZ\}$ is of the form
$$
[\mathcal{K}|_{U_Z}]=\begin{pmatrix} 
0 & -a(Z)\\
a(Z) & 0\end{pmatrix}
$$
for some $a(Z)\in\mathbb{R}$.  Hence, 
\begin{equation}
\label{eqS6thm}
\mathcal{K}|_{U_Z}=a(Z)\cdot J|_{U_Z},
\end{equation}
for all $Z\in T_pS^6$.  Now fix an orthonormal basis on $T_pS^6$ of the form
$$
e_1,~Je_1,~e_2,~Je_2,~e_3,~Je_3.
$$ 
Then $T_pS^6$ decomposes as 
\begin{equation}
\label{eqS6thmA}
T_pS^6=U_1\oplus U_2\oplus U_3,
\end{equation}
where $U_i:=U_{e_i}:=\mbox{span}\{e_i,Je_i\}$ for $i=1,2,3$.  Let $a_i:=a(e_i)$ for $i=1,2,3$.  Then 
\begin{equation}
\label{eqS6thmB}
\mathcal{K}|_{U_i}=a_i\cdot J|_{U_i},\hspace*{0.1in}i=1,2,3.
\end{equation}
We show that $a_1=a_2=a_3$.  Indeed, consider the element $Z=e_1+e_2+e_3$.  Then $Z\in U_Z:=\mbox{span}\{Z,JZ\}$ and
$$
\mathcal{K}Z=a(Z)JZ=a(Z)Je_1+a(Z)Je_2+a(Z)Je_3.
$$
On the other hand, by (\ref{eqS6thmB}), we have
$$
\mathcal{K}Z=\mathcal{K}e_1+\mathcal{K}e_2+\mathcal{K}e_3=a_1Je_1+a_2Je_2+a_3Je_3.
$$
From this, it follows immediately that
$$
a(Z)=a_1=a_2=a_3.
$$
Set $a:=a_1=a_2=a_3$.  Then (\ref{eqS6thmA}) and (\ref{eqS6thmB}) imply 
$$
\mathcal{K}=aJ.
$$
Consequently, for all $Z\in T_pS^6$, we have
\begin{equation}
\label{eqS6thm2}
(\nabla^g_ZJ)\circ \mathcal{K}=a(\nabla^g_ZJ)\circ J=-a J\circ (\nabla^g_ZJ)=-\mathcal{K}\circ (\nabla^g_ZJ).
\end{equation}
Combining (\ref{eqS6thm2}) with (\ref{eqS6thm1}) gives
$$
\mathcal{K}\circ (\nabla^g_ZJ)=a J\circ (\nabla^g_ZJ)=0.
$$
Since $J$ is invertible and $\nabla^g_ZJ\neq 0$ (since its kernel has dimension $2$ for nonzero $Z\in T_pS^6$), it follows that $a=0$.  Since $p\in S^6$ was arbitrary, we have $\mathcal{K}=0$.  The definition of $\mathcal{K}$ now implies that
$$
\psi J=-J\psi.
$$
From this, we conclude that $J^\psi :=\psi J\psi^{-1}=-J$, which is a contradiction as we had assumed $J^\psi$ to be integrable.  
\end{proof}

Recall that the standard almost complex structure $J$ on $S^6$ is induced from a choice of  $7$-dimensional cross product \cite{Calabi1958}
$$
\times : \mathbb{R}^7\times \mathbb{R}^7\rightarrow \mathbb{R}^7,\hspace*{0.1in} (u,v)\mapsto u\times v
$$
by defining $J_p: T_pS^6\rightarrow T_pS^6$ by
$$
J_p(v):=p\times v
$$
for all $p\in S^6$.  The 14-dimensional exceptional Lie group $G_2$ can be realized as the set of all invertible linear maps $F:\mathbb{R}^7\rightarrow \mathbb{R}^7$ which preserve the aforementioned cross product:
\begin{equation}
\label{eqG2Def}
G_2:=\{F\in \mathrm{GL}(\mathbb{R}^7)~|~F(u\times v)=(Fu)\times (Fv)\}.  
\end{equation}
From this definition, one can show that $G_2\subset \mathrm{SO}(7)$.  We refer the reader to \cite{Gray1969,Bryant1987,Karigiannis2020} for details.  Since $F\in G_2$ is a linear map, we have that 
$$
dF_p=F|_{T_pS^6}.
$$  
From the definition of $G_2$ and $J$, it follows that 
\begin{equation}
\label{eqG2}
dF_p\circ J_p=J_{F(p)}\circ dF_p,\hspace*{0.1in}\forall p\in S^6,~F\in G_2.
\end{equation}
Combining (\ref{eqG2}) with Theorem \ref{thmS6Twist}, one has the following extension:
\begin{corollary}
\label{corG2S6Twist}
Let $(g,J,\omega)$ be the standard nearly K\"{a}hler structure on $S^6$ and let $\times$ be the 7-dimensional cross product on $\mathbb{R}^7$ from which $J$ is defined.  Let $G_2\subset SO(7)$ be defined via (\ref{eqG2Def}).  Then $J^{\varphi}$ is nonintegrable for all $\varphi\in \mathrm{Aut}(TS^6)$ of the form  $\varphi:=dF\circ \psi \circ dF^{-1}$ where $\psi\in \mathrm{Aut}(TS^6)$ is a $g$-Codazzi map and $F\in G_2$.
\end{corollary}
\begin{proof}
Let $\varphi:=dF\circ \psi\circ dF^{-1}$ where $\psi\in \mathrm{Aut}(TS^6)$ is a $g$-Codazzi map and $F\in G_2$.  Then 
\begin{align*}
J^\varphi&:=\varphi\circ J\circ \varphi^{-1}\\
&=(dF\circ \psi \circ dF^{-1})\circ J\circ (dF\circ \psi^{-1}\circ dF^{-1})\\
&=dF \circ \psi \circ J\circ \psi^{-1} \circ dF^{-1}\\
&=dF \circ  J^\psi\circ dF^{-1},
\end{align*}
where the third equation follows from (\ref{eqG2}). The Nijenhuis tensor of $J^\varphi$ and $J^\psi$ are then related by
\begin{equation}
\label{eqJvarphiJpsi}
dF^{-1}(N_{J^\varphi}(X,Y))=N_{J^\psi}(dF^{-1}(X),dF^{-1}(Y)),\hspace*{0.1in}\forall~X,Y\in \mathfrak{X}(S^6).
\end{equation}
Since $J^\psi$ is nonintegrable by Theorem \ref{thmS6Twist}, (\ref{eqJvarphiJpsi}) implies that $J^\varphi$ is also nonintegrable.
\end{proof}

\noindent Theorem \ref{thmS6Twist} fully resolves Problem \ref{conjPsiTwist} for the special case of $g$-Codazzi maps.  For reasons of comparison, let us now turn our attention to a result of Bor and Hern\'{a}ndez-Lamoneda \cite{BorHernandezLamoneda1999}.  Let $h$ be any Riemannian metric on $S^6$.  In Corollary 3 of \cite{BorHernandezLamoneda1999}, the authors give sufficient conditions on $h$ so that any almost complex structure $I$ on $S^6$ which is compatible  with $h$ is necessarily nonintegrable.  These conditions are as follows:
\begin{itemize}
\item[(i)]$\mathcal{R}^h_p: \wedge^2 T^\ast_pS^6\rightarrow \wedge^2T^\ast_pS^6$ has positive eigenvalues for all $p\in S^6$.
\item[(ii)] $\lambda_{\mathrm{max}}/\lambda_{\mathrm{min}}<7/5$ where $\lambda_{\mathrm{max}}$ ($\lambda_{\mathrm{min}}$) is the largest (smallest) eigenvalue of $\mathcal{R}^h$ considered pointwise among all points of $S^6$.
\end{itemize}
While the result of  \cite{BorHernandezLamoneda1999} is quite nice, it has definite limitations when applied to Problem \ref{conjPsiTwist} for the special case of $g$-Codazzi maps.  Indeed, one can construct $g$-Codazzi maps where condition (ii) is no longer satisfied which, in turn, renders the result of  \cite{BorHernandezLamoneda1999} inapplicable.  We illustrate this with an  example.
\begin{example}
Let $f$ be a smooth function on $S^n$ for $n\ge 2$.  Let $\mathrm{Hess}_gf$ be the Hessian of $f$ with respect to the round metric $g$.  Let $\varphi: TS^n\rightarrow TS^n$ be the bundle map defined by 
$$
(\mathrm{Hess}_gf)(X,Y)=g(\varphi X,Y).
$$
Let $v\in T_pS^n$ be a unit vector and $\gamma_v^p(t)$ denote the geodesic satisfying $\gamma_v^p(0)=p$ and $\frac{d}{dt}|_{t=0}{\gamma}_v^p(0)=v$.  Explicitly,
$$
\gamma_v^p(t)=(\cos t)p+(\sin t)v.
$$
Let $E\subset TS^{n}$ denote the fiber bundle consisting of all tangent vectors of unit length, that is, $E$ is the sphere bundle of $S^n$.  Define $B: E\rightarrow \mathbb{R}$ by
$$
B(p,v):=\frac{d^2}{dt^2}\Big|_{t=0}f(\gamma_v^p(t)).
$$
One can show (cf Ch. 3 \cite{doCarmo1992}) that the minimum and maximum value that $B$ attains is precisely the minimum and maximum eigenvalues of $\varphi$ among all points of $S^n$.  

Now consider the Codazzi tensor $A_{f,c}$ on $S^6$ with $f=x_{1}x_{2}$ (where $x_i$ is the $i$-th natural coordinate function on  $\mathbb{R}^7$):
$$
A_{f,c}=\mathrm{Hess}_gf+(f+c)g.
$$ 
Define $F^c: TS^6\rightarrow TS^6$ by
$$
A_{f,c}(X,Y)=g(F^cX,Y).
$$
From this, we have
$$
F^c= \varphi+(f+c)\mbox{id}
$$
where $\varphi$ is again defined by
$$
(\mathrm{Hess}_gf)(X,Y)=g(\varphi X,Y).
$$
Note that $v\in T_pS^6$ is an eigenvector of $F_p^c: T_pS^6\rightarrow T_pS^6$ if and only if $v$ is an eigenvector of $\varphi_p: T_pS^6\rightarrow T_pS^6$.  Computing $B(p,v)$ for $f=x_1x_2$, one obtains
$$
B(p,v)=2(v_1v_2-p_1p_2),
$$
where $p=(p_1,\dots,p_7)$ and $v=(v_1,\dots, v_7)$.  The maximum and minimum values of $B(p,\cdot)$ taken over all $v\in T_pS^6$ such that $g(v,v)=1$  are found to be
$$
-3p_1p_2\pm \sqrt{(1-p_1^2)(1-p_2^2)}.
$$
The maximum and minimum eigenvalues of $F_p^c$ are then
$$
-2p_1p_2+c\pm \sqrt{(1-p_1^2)(1-p_2^2)}.
$$
Maximizing and minimizing the above expression subject to the condition $p_1^2+p_2^2\le 1$, one finds that the maximum and minimum eigenvalues of $F^c$ (among all points $p\in S^6$) is 
$$
c\pm \frac{3}{2}.
$$
Hence, $F^c$ is guaranteed to be invertible for $c>3/2$ and $c<-3/2$.  For $c>3/2$ ($c<-3/2$), $F_p^c$ has all positive (negative) eigenvalues for all $p\in S^6$.  With $c>3/2$ or $c<-3/2$, we obtain a $g$-Codazzi map $\psi^c$ by taking $\psi^c:=(F^c)^{-1}$ .  The maximum and minimum eigenvalues of $\psi^c$ are then
$$
\frac{1}{c-3/2},\hspace*{0.2in}\frac{1}{c+3/2}.
$$
Let $h^c:=g^{\psi_c}$.  By Proposition \ref{propCurvatureOperator}, the curvature operator 
$$
\mathcal{R}^{h^c}:\wedge^2 T^\ast S^6\rightarrow \wedge^2 T^\ast S^6
$$ 
is given by
$$
\mathcal{R}^{h^c}=\mathcal{R}^g\circ {\psi^c}^\ast ={\psi^c}^\ast.
$$
For $c>3/2$ and $c<-3/2$, the eigenvalues of $\mathcal{R}^{h^c}$ at all points of $S^6$ are positive.  For $c>3/2$, the maximum and minimum eigenvalues of $\mathcal{R}^{h^c}$ are 
$$
\lambda_{\mathrm{max}}=\frac{1}{(c-3/2)^2},\hspace*{0.1in}\lambda_{\mathrm{min}}=\frac{1}{(c+3/2)^2}.
$$
The condition $\lambda_{max}/\lambda_{min}<7/5$ (subject to the condition that $c>3/2$) is satisfied if
$$
c>9+\frac{3\sqrt{35}}{2}\approx 17.9.
$$
For $c<-3/2$, the maximum and minimum eigenvalues of $\mathcal{R}^{h^c}$ are respectively 
$$
\lambda_{\mathrm{max}}=\frac{1}{(c+3/2)^2},\hspace*{0.1in}\lambda_{\mathrm{min}}=\frac{1}{(c-3/2)^2}.
$$
The condition $\lambda_{max}/\lambda_{min}<7/5$ (subject to the condition that $c<-3/2$) is satisfied if
$$
c<-9-\frac{3\sqrt{35}}{2}\approx -17.9.
$$
Let $J$ be the standard almost complex structure on $S^6$ and let $I^c:=J^{\psi^c}$.   By Corollary 3 of \cite{BorHernandezLamoneda1999}, one concludes that $I^c$ is nonintegrable for $c>9+\frac{3\sqrt{35}}{2}\approx 17.9$ and $c<-9-\frac{3\sqrt{35}}{2}\approx -17.9$.  However, by Theorem \ref{thmS6Twist}, $I^c$ is nonintegrable for all $c$ such that $\psi^c:=(F^c)^{-1}$ exists.  In particular, Theorem \ref{thmS6Twist} and the above calculation implies that $I^c$ is actually nonintegrable for  $c>3/2$ and $c<-3/2$.
\end{example}
Theorem \ref{thmS6Twist} can be regarded as a reasonable first step in laying the foundation for future work on Problem \ref{conjPsiTwist}.  The form of this future work will entail solving Problem \ref{conjPsiTwist} for more general classes of automorphisms of $TS^6$, which may, in turn, yield new insight into Problem \ref{conjPsiTwist} and may (hopefully) lead to a complete solution to the problem.

\end{document}